\documentclass[a4paper,12pt]{article}
\usepackage{amsmath,amsthm,amssymb}
\usepackage[dvipsnames,svgnames,table]{xcolor}
\usepackage[plain]{fullpage}
\usepackage{enumerate} 
\usepackage{subcaption}

\usepackage{hyperref}
\usepackage{cleveref}
\hypersetup{
    pdftitle={Making an oriented graph acyclic using inversions of bounded or prescribed size},
    pdfauthor={Jørgen Bang-Jensen, Frédéric Havet, Florian Hörsch, Clément Rambaud, Amadeus Reinald, and Caroline Silva},
    colorlinks,
    linkcolor={RoyalBlue},
    citecolor={RubineRed},
    urlcolor={blue!80!black}
}

\usepackage{tikz}
\usepackage{authblk,enumitem}
\usepackage{thm-restate}
\usepackage{comment}
\usepackage{dsfont}

\newtheorem{theorem}{Theorem}[section]
\newtheorem{corollary}[theorem]{Corollary}
\newtheorem{proposition}[theorem]{Proposition}
\newtheorem{lemma}[theorem]{Lemma}
\newtheorem{claim}{Claim}[theorem]
\newtheorem{observation}[theorem]{Observation}

\theoremstyle{definition}

\newtheorem{problem}[theorem]{Problem}

\newenvironment{proofclaim}[1][]
    {\begin{proof}[Proof of the claim] }{\end{proof}}

\newcommand{\say}[1]{``#1''}

\newcommand{\FH}[1]{{\color{red}Fred: #1}}
\newcommand{\FHO}[1]{{\color{orange}Flo: #1}}
\newcommand{\FHOC}[1]{{\color{brown}Flocomment: #1}}

\definecolor{dark-green}{rgb}{0.2, 0.5, 0.2}
\newcommand{\ama}[1]{{\color{dark-green}Ama: #1}}

\newcommand{\Ra}{\Rightarrow}

\newcommand{\ora}[1]{\overrightarrow{#1}}

\renewcommand{\vec}[1]{\ora{#1}} 

\renewcommand{\leq}{\leqslant}

\renewcommand{\geq}{\geqslant}

\renewcommand{\epsilon}{\varepsilon}

\newenvironment{subproof}[1][]
	{\begin{proof}[Proof of the claim] }{\end{proof}}

\DeclareMathOperator{\Push}{Push}

\DeclareMathOperator{\Inv}{Inv}

\DeclareMathOperator{\inv}{inv}

\DeclareMathOperator{\fas}{fas}
\DeclareMathOperator{\bigO}{\mathcal{O}}
\DeclareMathOperator{\spn}{Span}

\DeclareMathOperator{\ct}{\ct}

\newcommand{\mcc}{\textsc{Multicoloured Clique}}
    
      \newcommand{\defdecproblem}[3]{
     \vspace{1mm}
     \noindent\fbox{
     \begin{minipage}{0.96\textwidth}
     \begin{tabular*}{\textwidth}{@{\extracolsep{\fill}}lr} #1 \\ \end{tabular*}
     {\bf{Input:}} #2 \\
     {\bf{Question:}} #3
     \end{minipage}
     }
     \vspace{1mm}
    }  
    
    \newcommand{\defparproblem}[4]{
     \vspace{1mm}
     \noindent\fbox{
     \begin{minipage}{0.96\textwidth}
     \begin{tabular*}{\textwidth}{@{\extracolsep{\fill}}lr} #1 \\ \end{tabular*}
     {\bf{Input:}} #2 \\
     {\bf{Parameter:}} #3 \\
     {\bf{Question:}} #4
     \end{minipage}
     }
     \vspace{1mm}
    }

\title{Making an oriented graph acyclic using inversions of bounded or prescribed size}

\author[1,6]{J\o{}rgen Bang-Jensen}
\author[2]{Fr\'ed\'eric Havet}
\author[3]{Florian Hörsch}
\author[2]{Cl\'ement Rambaud}
\author[4]{Amadeus Reinald}
\author[2,5]{Caroline Silva}

\affil[1]{SDU, Odense, Denmark}
\affil[2]{Universit\'e C\^ote d'Azur, CNRS, Inria, I3S, Sophia Antipolis, France}
\affil[3]{CISPA, Saarbrücken, Germany}
\affil[4]{LIRMM, Université de Montpellier, CNRS, Montpellier, France}
\affil[5]{Institute of Computing, UNICAMP, Campinas, Brazil}
\affil[6]{School of Mathematics, Shandong University, Jinan 250100, China}

\date{\today}

\begin{document}
\maketitle

\begin{abstract}
Given an oriented graph $D$, the inversion of a subset $X$ of vertices consists in reversing the orientation of all arcs with both endpoints in $X$.
When the subset $X$ is of size $p$ (resp. at most $p$), this operation is called an {\it $(=p)$-inversion} (resp. {\it $(\leq p)$-inversion}).
Then, an oriented graph is {\it $(=p)$-invertible} if it can be made acyclic by a sequence of $p$-inversions.
We observe that, for $n=|V(D)|$, deciding whether $D$ is $(=n-1)$-invertible is equivalent to deciding whether $D$ is acyclically pushable, and thus NP-complete.
In all other cases, when $p \neq n-1$, we construct a polynomial-time algorithm to decide $(=p)$-invertibility.

We then consider the {\it $(= p)$-inversion number}, $\inv^{= p}(D)$ (resp. {\it $(\leq p)$-inversion number}, $\inv^{\leq p}(D)$), defined as the minimum number of $(=p)$-inversions (resp. $(\leq p)$-inversions) rendering $D$ acyclic.
We show that every $(=p)$-invertible digraph $D$ satisfies $\inv^{= p}(D) \leq |A(D)|$ for every integer $p\geq 2$.
When $p$ is even, we bound $\inv^{= p}$ by a (linear) function of the feedback arc set number, and rule out the existence of any bounding function for odd $p$.

Finally, we study the complexity of deciding whether the $(= p)$-inversion number, or the $(\leq p)$-inversion number, of a given oriented graph is at most a given integer $k$.
For any fixed positive integer $p \geq 2$, when $k$ is part of the input, we show that both problems are NP-hard even in tournaments.
In general oriented graphs, we prove $W[1]$-hardness for both problems when parameterized by $p$, even for $k=1$.
In contrast, we exhibit polynomial kernels in $p + k$ for both problems in tournaments.

\medskip
    \noindent{}{\bf Keywords:}  inversion ; orientation; NP-hardness; acyclicity ; reconfiguration.
\end{abstract}

\section{Introduction}

Problems concerned with rendering a digraph acyclic through a small number of operations are among the most studied in digraph theory.
A prominent example are {\bf feedback arc sets}, which for a digraph $D$, are the subsets of arcs whose deletion makes $D$ acyclic.
The corresponding parameter is the {\bf feedback-arc-set number} $\fas(D)$ of $D$, defined as the minimum size of a feedback arc set.
Algorithmically, the computation of $\fas(D)$ is the {\sc Feedback Arc Set} problem, already appearing in Karp's seminal list of 21 NP-hard problems~\cite{karp1972}.
A well-studied restriction is {\sc Feedback Arc Set} on tournaments, which was shown to be NP-complete by Alon~\cite{alonSJDM20} and Charbit, Thomass\'e, and Yeo~\cite{charbitCPC16}.
Nevertheless, this restriction admits a polynomial-time approximation scheme, due to Kenyon and Schudy~\cite{kenyon2007rank}, and linear kernels of size $(2+\epsilon)k$ for every $\epsilon >0$, as shown by Bessy et al.~\cite{bessyKernel}.
Asymptotic considerations of $\fas$ were first brought up by Erdös and Moon~\cite{erdosSetsConsistentArcs1965}, who looked at estimating $\fas(n)$, the maximum value of $\fas$ over digraphs of order $n$.
This question was eventually settled thanks to results by Spencer~\cite{spencer1971,spencer1980} and de~la~Vega~\cite{delavega1983}, achieving to show $\inv^{\leq 2}(n) = \frac{1}{2}\binom{n}{2} -\Theta(n^{3/2})$.
In this paper, we investigate analogous questions when looking to render a digraph acyclic through \emph{inversions of bounded size}, which are operations generalizing arc deletions.

Observe that if $F$ is a minimum feedback arc set in a digraph $D=(V,A)$, then we may obtain an acyclic digraph from $D$ by either removing the arcs of $F$, or by reversing each of them: replacing each $uv\in F$ with $vu$.
In this sense, {\sc Feedback Arc Set} can be rephrased as a problem about arc reversals.
Lifting this, Belkhechine et al.~\cite{BBBP10} introduced inversions, allowing for the \say{reversal} of any induced subdigraph.
Formally, if $X$ is a set of vertices of $D$, the {\bf inversion} of $X$ consists in reversing the orientation of all arcs with both endvertices in $X$.
Then, the {\bf inversion number} $\inv(D)$ of a digraph $D$, is the minimum number of inversions needed to turn $D$ acyclic.
Observe that questions about $\inv(D)$ are only of interest for oriented graphs — digraphs forbidding digons (cycles of length two).
Indeed, if a digon $uv,vu$ exists in $D$, it will be preserved by any inversion, meaning $\inv(D) = + \infty$.
On the other hand, if $uv$ is not part of a digon, then inverting $\{u,v\}$ corresponds to the reversal of $uv$. In particular, a digraph can be rendered acyclic by a sequence of inversions if and only if it is oriented.

Following the well-established literature on feedback arc sets, inversion number has been studied both algorithmically and asymptotically.
On the algorithmic side, the authors of~\cite{BBBP10} proved that determining whether a given tournament $T$ satisfies $\inv(T) \leq k$ is in XP, that is, it is polynomial-time solvable for every fixed $k$.
This was improved upon by Alon et al.~\cite{APSSW}, who showed that the problem is fixed-parameter tractable (FPT) with respects to $k$, meaning it admits an $f(k)n^{O(1)}$ algorithm.
In contrast, Bang-Jensen et al.~\cite{BCH} proved that given oriented graph $D$, deciding $\inv(D) \leq k$ is NP-complete for $k \in \{1,2\}$. This was generalized to all $k$ by Alon et al.~\cite{APSSW}.
On the asymptotic side, the main problem has been to estimate $\inv(n)$, the maximum inversion number over all oriented graphs of order $n$.
Independently, Aubian et al.~\cite{inversion} and Alon et al.~\cite{APSSW} proved $n - 2\sqrt{n\log n} \leq \inv(n) \leq n - \lceil \log (n+1) \rceil$.

In this paper, we will be interested in finer variants of inversions by bounding (or prescribing) the size of the inverted sets.
Inversions of bounded size were recently introduced by Yuster~\cite{yuster2023tournamentinversion} as {\bf $(\leq p)$-inversions}, which are inversions over sets of size at most $p$.
These operations range from arc reversals when $p=2$, allowing us to express feedback arc sets, to general inversions when $p$ exceeds the order $n$.
Given a digraph $D$, the parameter of interest is the {\bf  $(\leq p)$-inversion number}, $\inv^{\leq p}(D)$, defined as the minimum number of $(\leq p)$-inversions rendering $D$ acyclic.
Yuster investigated the asymptotic behaviour of $\inv^{\leq p}(n)$, defined as the maximum of $(\leq p)$-inversion numbers over all oriented graphs of order $n$.
Observe that $\inv^{\leq 2}(D) =\fas(D)$, meaning all results on $\fas$ apply for $p=2$.
Asymptotically, this means we already know that $\inv^{\leq 2}(n) = \frac{1}{2}\binom{n}{2} -\Theta(n^{3/2})$ by~\cite{delavega1983}.
Yuster~\cite{yuster2023tournamentinversion} proved
$\frac{1}{12}n^2 - o(n^2) \leq \inv^{\leq 3}(n) \leq \frac{257}{2592}n^2 + o(n^2)$ 
and conjectured 
$\inv^{\leq 3}(n) = \frac{1}{12}n^2 + o(n^2)$.
He also proved that there exists a positive integer $p_0$ such that,
for every integer $p$ with $p \geq p_0$, there exist $\delta_p, \epsilon_p>0$ such that
\[
    \left(\frac{1}{2p(p-1)}+\delta_p\right)n^2 - o(n^2) \leq \inv^{\leq p}(n)\leq \left( \frac{1}{2\lfloor p^2/2\rfloor} -\epsilon_p\right)n^2 + o(n^2)
\]
as $n$ tends to $+\infty$.
Still, the computational aspects of $\inv^{\leq p}$ remained largely open until now.

\subsection{Overview of the results}

In this paper, we consider both inversions of bounded size, and those of prescribed size, for which we initiate the study. An inversion of prescribed size, specifically a {\bf  $(= p)$-inversion}, is the inversion of a set size exactly $p$.
Then, the {\bf  $(= p)$-inversion number} $\inv^{= p}(D)$ of an oriented graph $D$ as the minimum number of $(=p)$-inversions rendering $D$ acyclic.
In~\Cref{sec:invertibility}, we address the question of $(=p)$-invertibility, that is, deciding whether $\inv^{= p}(D) < \infty$.
Then, \Cref{sec:bounds} deals with asymptotic bounds on $\inv^{=p}$ with respects to the order $n$, as well as $\inv^{\leq p}$ and $\fas$.
Finally, in \Cref{sec:complexity} we investigate the (parameterized) complexity of computing both $\inv^{\leq p}$ and $\inv^{= p}$, in particular for tournaments.
In the following, we motivate all problems considered and survey our results for each section.

\subsubsection{Inversions of prescribed size}

\Cref{sec:invertibility} addresses arguably the first natural question arising when considering $(= p)$-inversions: given an oriented graph, is it {\bf $(=p)$-invertible} at all? That is, can it be made acyclic using (any number of) $(=p)$-inversions? 
Clearly, this is not the case for oriented graphs of order at most $p$ that are not already acyclic. But this is not the only obstacle: for every odd $p$, there exist oriented graphs of arbitrarily large order which are not $(=p)$-invertible.
A simple invariant witnessing this when $p$ is odd, is given by the parity of the out-degrees of the vertices of a tournament, which is preserved under $(=p)$-inversions.
Hence, to be $(=p)$-invertible, a tournament of order $n$ must have exactly as many vertices of even out-degree as $TT_n$, the transitive tournament of order $n$, that is $\lceil n/2\rceil$. 

We exhibit a dichotomy on the complexity of $(=p)$-invertibility: given an $n$-vertex oriented graph $D$ and a positive integer $p$, deciding whether $D$ is $(=p)$-invertible is polynomial, except if $p=n-1$, in which case this problem is NP-hard.
First, in Subsection~\ref{subsec:NP}, we show that {\sc $(n-1)$-Invertibility} is equivalent to decide whether a given digraph $D$ is acyclically pushable. It is thus NP-complete as shown by Klostermeyer~\cite{Klo99}. 


Next, the remainder of Section~\ref{sec:invertibility} is devoted to prove the polynomial-time solvability when $p\neq n-1$.
First, in Subsection~\ref{subsec:carac}, we characterize the pairs of tournaments on the same vertex set that can be transformed from one to the other using $(=p)$-inversions.
This is done by using the notion of {\bf forward arcs}, which for a (ordered) tournament $T$ on $[n]$, are the arcs $ij$ with $i<j$. We denote by $F(T)$ the set of forwards arcs in $T$ and by $F(T,i)$ the set of forwards arcs in $T$ incident with $i$ and by $F^+(T,i)$ the set of forwards arcs in $T$ with tail $i$.

\begin{theorem}\label{thm:characterization_of_=p_equivalent_tournaments}
    Let $n,p$ be positive integers with $n \geq p+2$,
    and let $T_1,T_2$ be two tournaments with vertex set $[n]$.
    There exists a family $X_1, \dots, X_\ell \subseteq [n]$ such that $T_2 = \Inv(T_1; X_1, \dots, X_\ell)$ if and only if
    one of the following occurs
    \begin{itemize}
        \item $p = 0 \mod 4$ and $|F(T_1)| = |F(T_2)| \mod 2$,
            
        \item $p = 1 \mod 4$, $|F(T_1,i)| = |F(T_2,i)| \mod 2$ for every $i \in [n-1]$ and $|F(T_1)| = |F(T_2)|$. 
        \item $p = 2 \mod 4$, or
        \item $p = 3 \mod 4$, and $|F(T_1,i)| = |F(T_2,i)| \mod 2$ for every $i \in [n-1]$.
    \end{itemize}
\end{theorem}

This immediately yields the following algorithmic result.
\begin{corollary}\label{cor:equal_p_invertibility_is_polynomial}
   Given an integer $p \geq 2$ and two tournaments $T_1, T_2$ with vertex set $[n]$, 
   where  $n \geq p+2$, one can decide in linear time whether there exists a $(=p)$-family $\mathcal{X}$ of subsets of $[n]$ such that $\Inv(T_1; \mathcal{X}) = T_2$.
\end{corollary}

Based on \Cref{thm:characterization_of_=p_equivalent_tournaments}, 
we prove the following characterisation of $n$-vertex tournaments which are $(=p)$-invertible.

\begin{theorem}\label{thm:characterization_of_=p_invertible_tournaments}
    Let $p \geq 2$ and let $n \geq p+2$. Let $T$ be a tournament with vertex set $[n]$.
    The tournament $T$ is $(=p)$-invertible if and only if 
    $p$ is even or $|\{i \in [n] \mid d^+(i) \text{ even}\}| = \lceil \frac{n}{2} \rceil$.
\end{theorem}

Then, in Subsection~\ref{subsec:algo-gen}, we prove the following algorithmic theorem for general oriented graphs.

\begin{restatable}{theorem}{algo}\label{algo}
    There is a polynomial-time algorithm that, given a positive integer $p$ and an oriented graph $D$ of order distinct from $p+1$, decides whether $D$ is $(=p)$-invertible. 
\end{restatable}   

The dichotomy between the even and odd cases for $p$, outlined in Theorem~\ref{thm:characterization_of_=p_invertible_tournaments} for tournaments, extends to oriented graphs.
Indeed, while there exist non-$(=p)$-invertible oriented graphs of every order at least 3 when $p$ is odd, every oriented graph of order at least $p+2$ is $(=p)$-invertible when $p$ is even.
\begin{restatable}{theorem}{algoeven}\label{algoeven}
Let $p$ be a positive even integer. 
Every oriented graph of order at least $p+2$ is $(=p)$-invertible.
\end{restatable}


\subsubsection{Asymptotic bounds}

Once we know that an oriented graph $D$ is $(=p)$-invertible, we can ask the question of how many $(=p)$-inversions are needed to render it acyclic, that is its {\bf  $(= p)$-inversion number}, denoted by $\inv^{= p}(D)$, is the minimum number of $(=p)$-inversions rendering $D$ acyclic, or $+\infty$ if $D$ is not $(=p)$-invertible.
Clearly, $\inv^{= p}(D)\geq \inv^{\leq p}(D)$ and $\inv^{= 2}(D)=\inv^{\leq 2}(D) =\fas(D)$. 
For every integer $n$, we let $\inv^{= p}(n)$ be the maximum of the $(= p)$-inversion number over all the $(=p)$-invertible oriented graphs of order $n$. 

Our next contribution, in Section~\ref{sec:bounds}, is to provide upper bounds on the $(=p)$-inversion number.
We first prove in Subsection~\ref{ssec:bound-size} an upper bound in terms of the size of the oriented graph.
\begin{theorem}\label{thm:borneA}
    Let $D$ be an oriented graph and let $p$ be an integer.
    If $D$ is $(=p)$-invertible, then $\inv^{= p}(D) \leq |A(D)|$.    
\end{theorem}

We then restrict our study to the case of {\bf even $p$}. Then, recall that all the oriented graphs of order at least $p + 2$ are $(=p)$-invertible by Theorem~\ref{algoeven}.
We establish a collection of incomparable upper bounds on the $(= p)$-inversion number as a linear function of the feedback-arc-set number and a corrective factor depending on the order.
The smaller the linear function, the larger the corrective factor.
\begin{theorem}\label{thm:upper-fas}
    Let $p\geq 4$ be an even integer and $D$ an oriented graph of order $n$. Then
    \[
        \inv^{= p}(D) \leq 
        \left\{
        \begin{array}{lrlll}
            (i)   & (2p -2) \fas(D) &  + \bigO(1) & \text{if $n\geq p+2$} &\text{\rm(Theorem~\ref{thm:ub-fas}),} \\
            (ii)  & 2  \fas(D) & + \bigO(n) & \text{if $n\geq p+2$}& \text{\rm(Theorem~\ref{thm:bound-2fas}),} \\
            (iii) &  \fas(D) & + \bigO(n^{1+1/k}) & \text{if $n\geq p+2k+2$}&  \text{\rm(Theorem~\ref{thm:bound-fas}),} \\
            (iv)  &   \frac{1}{\lceil p/2\rceil \cdot \lfloor p/2\rfloor}  \fas(D) & + o(n^2) & & \text{\rm(Theorem~\ref{thm:upper-opt}).}
        \end{array}\right.     
    \]
\end{theorem}

\noindent Observe that the bound (iv) in this theorem is asymtotically tight. Indeed, if $D$
is a bipartite oriented graph, then each $(=p)$-inversion reverses at most $\lceil p/2\rceil \cdot \lfloor p/2\rfloor$ arcs, and so $\inv^{= p}(D) \geq \frac{1}{\lceil p/2\rceil \cdot \lfloor p/2\rfloor}  \fas(D)$.

\medskip

The upper bound $(i)$ of \Cref{thm:upper-fas} has several consequences. 
The first one is that $\inv^{=p}(n)$ and $\inv^{\leq p}(n)$ are close to each other when $p$ is even.
\begin{restatable}{corollary}{closeinv}\label{cor:invleqp_and_inveqp}
 Let $p$ be an even positive integer. Then
    $\inv^{=p}(n) = \inv^{\leq p}(n) + \Theta(n)$.
\end{restatable}
\noindent
Together with Yuster's result, this implies 
$\displaystyle \inv^{= p}(n)\leq \left( \frac{1}{p^2} -\epsilon_p\right)n^2 + o(n^2)$ for some positive $\epsilon_p$ when $p$ is even and at least $4$.

The second consequence of $(i)$ is that, for any fixed even positive integer $p$,
$\inv^{=p}$ and $\inv^{\leq p}$ are (linearly)  tied.
\begin{restatable}{corollary}{evenbound}\label{cor:even-bound}
 Let $p \geq 4$ be an even integer, and let $D$ be an oriented graph with at least $p + 2$ vertices and at least one directed cycle. Then, 
    $$ \inv^{\leq p}(D) \leq \inv^{= p}(D) \leq (4p-4)\binom{p}{2} \inv^{\leq p}(D).$$\end{restatable}

A third consequence is the following.
\begin{restatable}{theorem}
{borneA}\label{thm:borneAbis}
 Let $p \geq 4$ be an even integer and let $D$ be an oriented graph of order $n \geq p + 2$. Then, 
 $$\inv^{= p}(D) \leq \frac{|A(D)|}{p-1} + 2p^2 n.$$
\end{restatable}  
\noindent This improves on Theorem~\ref{thm:borneA} for oriented graphs with a superlinear number of arcs.

In Subsection~\ref{ssec:bound-fas-odd}, we investigate whether analogous results to the ones obtained for even values of $p$ hold when $p$ is odd.
It turns out that this is not the case for the upper bound (i) of Theorem~\ref{thm:upper-fas}: in Theorem~\ref{thm:no-f}, we prove that the $(=p)$-inversion number cannot be upper bounded by any function of the feedback arc-set number.
This implies that, unlike in the case where $p$ is even, $\inv^{=p}$ and $\inv^{\leq p}$ are not tied when $p$ is odd.
This further accentuates the dichotomy between even and odd values of $p$.
The existence of upper bounds on the $(=p)$-inversion number analogous to the bounds (ii) to (iv) of Theorem~\ref{thm:upper-fas} is left as open problems.

\subsubsection{Complexity}

Finally, in Section~\ref{sec:complexity}, we study the complexity of computing the $(\leq p)$-inversion number and the $(= p)$-inversion number.
Specifically, we are interested in the two following general problems, taking $k$ and $p$ as inputs, and some of their restrictions.

\defdecproblem{\sc Prescribed Size Inversion}{An oriented graph $D$ and two positive integers $p$ and $k$}{Does $D$ admit a decycling $(= p)$-family of size at most $k$?}

\defdecproblem{\sc Bounded Size Inversion}{An oriented graph $D$ and two positive integers $p$ and $k$}{Does $D$ admit a decycling $(\leq p)$-family of size at most $k$?}

Fixing $p$ yields the following restrictions.

\defdecproblem{\sc $(= p)$-Inversion}{An oriented graph $D$ and a positive integer $k$}{Does $D$ admit a decycling $(= p)$-family of size at most $k$?}

\defdecproblem{\sc $(\leq p)$-Inversion}{An oriented graph $D$ and a positive integer $k$}{Does $D$ admit a decycling $(\leq p)$-family of size at most $k$?}

The restriction of the above problems to instances in which the input oriented graph is a tournament are named with a  {\sc Tournament} preceding the name of the general problem. For example, 
the restriction of {\sc $(\leq p)$-Inversion} to tournaments is named {\sc Tournament $(\leq p)$-Inversion}.

Note that {\sc $(= 2)$-Inversion} and {\sc $(\leq 2)$-Inversion} correspond to the problem of computing the feedback-arc-set number of an oriented graph, which we recall is already NP-complete, even for tournaments.
In Subsection~\ref{subsec:NP-hardness}, we show that {\sc Tournament $(= p)$-Inversion} and {\sc Tournament $(\leq p)$-Inversion} are also NP-complete for any fixed $p$ greater than $2$. (Theorem~\ref{rzgu}). 

On the other hand, it is not difficult to see that deciding whether the $(\leq p)$-inversion number (resp. $(= p)$-inversion number) of an oriented graphs is at most $k$ can be solved in polynomial time  when both $k$ and $p$ are fixed. In light of these results, we investigate the parameterized complexity of the problems.

We begin in Subsection~\ref{subsec:W1} by showing the following hardness result.
\begin{restatable}{theorem}{Wone}\label{thm:W1}
    Given an oriented graph $D$ and an integer $p$, the following problems are $W[1]$-hard parameterized by $p$:
    \begin{itemize}
        \item deciding whether $D$ can be made acyclic by a single $(\leq p)$-inversion,
        \item deciding whether $D$ can be made acyclic by a single $(=p)$-inversion.
    \end{itemize}
\end{restatable}
This theorem implies that \textsc{Prescribed Size Inversion} and  \textsc{Bounded Size Inversion} are $W[1]$-hard when parameterized by $p$ even for $k=1$. Therefore, it is  $W[1]$-hard when parameterized by $p+k$.

In contrast, we prove that both problems admit a linear kernel:
\begin{restatable}{theorem}{kernel}\label{thm:kernel}
    For any fixed $\epsilon > 0$, {\sc Tournament Bounded Size Inversion} and {\sc Tournament Prescribed Size Inversion} admit a kernel of size at most $(1+ \epsilon)k^2 p^3$.
\end{restatable}
As a consequence, these problems are FPT when parameterized by $k + p$.
It is also worth noting that the problem of deciding whether a given tournament $T$ admits a decycling family (without restriction on the size of its members) of size at most $k$ has been shown to be FPT when parameterized by $k$ by Alon et al.~\cite{APSSW}. Nevertheless, the existence of a polynomial kernel for this problem is still open.

It is natural to ask about the complexity of {\sc Prescribed Size Inversion} and {\sc Bounded Size Inversion} when $p$ is fixed and $k$ is taken as the parameter. This corresponds to studying the complexity of {\sc $(= p)$-Inversion} and {\sc $(\leq p)$-Inversion} parameterized by $k$.
When $k = 2$, both problems are {\sc Feedback Arc Set}, which is known to be FPT by a result of Chen et al.~\cite{fas-FPT}.
The case of larger values of $k$ remains open.

\begin{problem}
For any integer $p\geq 3$, is {\sc $(= p)$-Inversion} (resp. {\sc $(\leq p)$-Inversion}) fixed-parameter tractable when parameterized by $k$?
\end{problem}

\section{Notations and definitions}

Digraph notation not given below is consistent with \cite{bang2009}.
For an oriented graph $D$, let $\sigma=(v_1,v_2, \ldots, v_n)$ be an ordering of the vertices of $D$. An arc $v_iv_j$ is {\bf forward} (according to $\sigma$) if $i<j$ and {\bf backward} (according to $\sigma$) if $j<i$.
It is well-known that an oriented graph is acyclic if and only if it admits an {\bf acyclic ordering} (sometimes called topological ordering), that is, an ordering of its vertices with no backward arcs.

An {\bf $(=p)$-set} (resp. a {\bf $(\leq p)$-set}) is a set of cardinality exactly $p$ (resp. at most $p$). 
Let $D$ be an oriented graph and $\cal X$ be a family of subsets of $V(D)$.
We say that $\cal X$ is a {\bf $(=p)$-family} (resp. {\bf $(\leq p)$-family}) if all members of $\cal X$ are $(=p)$-sets (resp. $(\leq p)$-sets).
We denote by $\Inv(D; {\cal X})$ the oriented graph obtained after inverting all sets of $\cal X$ one after another. Observe that this is independent of the order in which we invert those sets: $\Inv(D; {\cal X})$ is obtained from $D$ by reversing exactly those arcs for which an odd number of members of ${\cal X}$ contain both endvertices.
If ${\cal X} = \{X\}$ for a set $X\subseteq V(D)$, then we write $\Inv(D, X)$ for $\Inv(D; {\cal X})$.
A {\bf decycling family} of an oriented graph $D$ is a family ${\cal X}$ of subsets of $V(D)$ such that $\Inv(D; {\cal X})$ is acyclic.

\section{\texorpdfstring{$(=p)$}{(=p)}-invertibility}\label{sec:invertibility}

In this section, we deal with our results describing which oriented graphs can be made acyclic by $(=p)$-inversions for some positive integer $p$.
First, in Section \ref{subsec:NP}, we argue that $(n-1)$-invertibility is already known to be NP-complete.
Next, in Section \ref{subsec:carac}, we give our characterization for $(=p)$-invertibility for tournaments in all cases, that is, we prove \Cref{thm:characterization_of_=p_invertible_tournaments}. As a side product, we also prove \Cref{thm:characterization_of_=p_equivalent_tournaments}.  Finally, in Section \ref{subsec:algo-gen}, we use this result to obtain \Cref{algo}, the corresponding algorithmic result in arbitrary oriented graphs, and \Cref{algoeven}.

\subsection{NP-completeness of {\sc \texorpdfstring{$(n-1)$}{(n-1)}-Invertibility}}\label{subsec:NP}

Let $D$ be an oriented graph and $X\subseteq V(D)$. We define $\Push(D,X)$ to be the oriented graph obtained from $D$ by reversing the orientation of all arcs with exactly one endvertex in $X$. We say
the vertices of $X$ are {\bf pushed} and that $\Push(D,X)$ is the result of {\bf pushing $X$ in $D$}. 
This operation has been studied in many papers~\cite{FiRy95,Klo98,Klo99,KlSv98,KlMc04,McWo00,HMW01,HuWo01,HMY02,RiZ06,HeHu09}.
An oriented graph is {\bf acyclically pushable} if there exists a set $X$ of vertices such that $\Push(D,X)$ is acyclic. Observe that $\Push(D,\{x\})$ is obtained from $\Inv(D,V(D)\setminus \{x\})$ by reversing every arc. 
Thus an oriented graph on $n$ vertices is acyclically pushable if and only if it is $(=n-1)$-invertible.
Klostermeyer~\cite{Klo99} proved that deciding whether a given oriented graph is
acyclically pushable is NP-complete, and 
Huang, McGillivray and Yeo~\cite{HMY02} showed it remains NP-complete when restricted to bipartite graphs.
This directly implies the following.

\begin{theorem}
    {\sc $(n-1)$-Invertibility} is NP-complete, even for oriented bipartite graphs. 
\end{theorem}

\subsection{Characterization of \texorpdfstring{$(=p)$}{(=p)}-invertible tournaments}\label{subsec:carac}

In this section, we prove \Cref{thm:characterization_of_=p_equivalent_tournaments,thm:characterization_of_=p_invertible_tournaments}.
The proof splits into four cases, depending on the value of $p$ modulo $4$.

\medskip

Let $p,n$ be integers with $p \geq 2$ and $n \geq p+2$. 
When we want to go from a tournament with vertex set $[n]$ to another tournament with the same vertex set  using $(=p)$-inversions, several obstructions appear, depending on the value of $p$ modulo $4$.
When $p$ is odd, for any vertex $v$, $d^+(v) \bmod 2$ is invariant under $(=p)$-inversions.
Moreover, when $\binom{p}{2}$ is even (that is when $p= 0 \bmod 4$ or $p=1 \bmod 4$), the parity of the number of backward arcs with respect to the natural ordering (the arcs $ji$ with $i<j$) is also invariant
under $(=p)$-inversions.
We will actually show that these two invariants are essentially the only ones, which gives a characterization
of tournaments reachable from any given tournament via $(=p)$-inversions.
To prove that, we will use an algebraic representation of the problem.

We shall do all operations in the two-element field $\mathbb{F}_2$. For convenience, we index the vectors of the vector space 
 $\mathbb{F}_2^{\binom{n}{2}}$ by the $(=2)$-sets of $[n]$.
Thus the coordinates of a vector $\mathbf{u}\in \mathbb{F}_2^{\binom{n}{2}}$ are named as $\mathbf{u}_{\{i,j\}}$ for every pair of distinct integers $i,j$ in $[n]$. 

We start by defining a representation of (labelled) $n$-vertex tournaments by $\binom{n}{2}$-dimensional vectors over $\mathbb{F}_2$.
Let $T$ be a tournament with vertex set $[n]$.
We define $\mathbf{a}_T \in \mathbb{F}_2^{\binom{n}{2}}$ by
\[
   (\mathbf{a}_T)_{\{i,j\}} =
   \begin{cases}
       1 &\textrm{if $ji \in A(T)$,} \\
       0 &\textrm{if $ij \in A(T)$,}
   \end{cases}
\]
for every $i,j \in [n]$ with $i<j$.
Let $X \subseteq [n]$, and let $\mathbf{i}_X \in \mathbb{F}_2^{\binom{n}{2}}$ be defined by
\[
    (\mathbf{i}_X)_{\{i,j\}} =
   \begin{cases}
       1 &\textrm{if $i,j \in X$,} \\
       0 &\textrm{otherwise,}
   \end{cases}
\]
for every $i,j \in [n]$ with $i\neq j$.
The point of these definitions is the following observation.
\begin{observation}\label{esdtrfgzhu}
    For every tournament $T$ on $[n]$ and for every $X \subseteq [n]$,
    \[
    \mathbf{a}_{\Inv(T;X)} = \mathbf{a}_T + \mathbf{i}_X.
    \]
\end{observation}

The proof of Theorem~\ref{thm:characterization_of_=p_equivalent_tournaments} will consist (for $p$ and $n$ fixed) in defining a linear mapping $\pi \colon \mathbb{F}_2^{\binom{n}{2}} \to \mathbb{F}_2^t$
for some $t > 0$ such that for all tournaments $T',T$ on $[n]$:
\begin{equation}\label{eq:pi_is_correct}
    \begin{array}{c}
        \text{there exists an $(=p)$-family $\mathcal{X} \subseteq 2^{[n]}$ such that $\Inv(T;\mathcal{X})=T'$} \\
        \text{if and only if $\pi(\mathbf{a}_T) = \pi(\mathbf{a}_{T'})$.}
    \end{array}
\end{equation}
Informally, $\pi$ corresponds here to the most general quantity which is invariant under $(=p)$-inversions.

Equation \eqref{eq:pi_is_correct} can be reformulated as follows.
For every $\mathbf{u} \in \mathbb{F}_2^{\binom{n}{2}}$, with $\mathbf{0}$ being the null element of $\mathbb{F}_2^t$: 
\begin{equation}\label{eq:pi_is_correct_linear_algebra}
    \text{$\mathbf{u} \in \spn \{\mathbf{i}_X \mid X \subseteq [n], |X|=p\}$ if and only if $\pi(\mathbf{u})=\mathbf{0}$.}
\end{equation}

We construct the linear $\pi$ according to the value of $p$ modulo $4$, in order to satisfy~\eqref{eq:pi_is_correct_linear_algebra}.
Let us begin by showing a few lemmas routinely used across all cases.

\subsubsection{Preliminaries}

\begin{proposition}\label{claim:4cycle}
    Let $D$ be the orientation of a graph $G$ on $n$ vertices and let $p$ be an integer with $n-2 \geq p \geq 2$.
    Let $x_0,x_1,x_2,x_3$ be any four vertices of $G$. Then using four $(=p)$-inversions, we can reverse the orientation of the edges in $\{x_0x_1, x_1x_2, x_2x_3, x_3x_0\}\cap E(G)$ and none other.
\end{proposition}
\begin{proof}
    Let $X$ be a set of $p-2$ vertices of $V(G)\setminus\{x_0,x_1,x_2,x_3\} $. Observe that such a set exists as $n \geq p+2$.
    Let us invert the sets $X\cup \{x_0, x_1\}$,  $X\cup \{x_1, x_2\}$, $X\cup \{x_2, x_3\}$, and $X\cup \{x_3, x_0\}$.
    Doing so, all arcs are reversed an even number of times except those with underlying edge in  $\{x_0x_1, x_1x_2, x_2x_3, x_3x_0\}$.
\end{proof}

\begin{proposition}\label{claim:2adjacentedges}
    Let $D$ be an orientation of a graph $G$ on $n$ vertices and and let $p$ be an even integer with $n-2 \geq p \geq 2$. 
    With $2p-2$ $(=p)$-inversions, we can reverse the orientation of a given pair of adjacent edges of $G$, but no other edges. 
\end{proposition}
\begin{proof}
    Let $uv_1$ and $uv_2$ be two adjacent edges of $G$ and $X$ a set of $p-2$ vertices in $V(G)\setminus \{u, v_1, v_2\}$. Observe that such a set exists as $n \geq p+2$.
    Invert first  $X\cup \{u, v_1\}$ and $X\cup \{u, v_2\}$.
    This reverses the orientation of the edges $uv_1$, $uv_2$, and those of the complete bipartite graph with bipartition $(\{v_1,v_2\},X)$.
    Since $p$ is even, the edge set of this bipartite graph can be decomposed into $(p-2)/2$ $4$-cycles of $G$. Hence, they can be reversed one after another by Proposition~\ref{claim:4cycle} using four inversions each.
    Doing so, only the arcs $uv_1$ and $uv_2$ remain reversed, and the total number of inversions is $4(p-2)/2+2=2p-2$.    
\end{proof}

\begin{proposition}\label{claim:2edges}
    Let $D$ be an oriented graph on $n$ vertices and and let $p$ be an even integer with $n-2 \geq p \geq 2$. 
    With $4p-4$ $(=p)$-inversions, we can reverse a given pair of  non-adjacent arcs and no others.
\end{proposition}
\begin{proof}
    To reverse two non-adjacent arcs $u_1v_1$ and $u_2v_2$, 
    one can use Proposition~\ref{claim:2adjacentedges} to reverse $u_1v_1$ along with $u_1u_2$ and then $u_1u_2$ with $u_2v_2$.
    This can be done with $2 (2p-2)= 4p-4$ $(=p)$-inversions. 
\end{proof}

\subsubsection{Case \texorpdfstring{$p = 2 \mod 4$}{p = 2 mod 4}.}

Let $n$ be a positive integer, and let $p$ be a positive integer such that $n \geq p+2$
and $p = 2 \mod 4$.
We take $t=0$, and so
for every $\mathbf{u} \in \mathbb{F}_2^{\binom{n}{2}}$, $\pi(\mathbf{u}) = \mathbf{0}$, here taking the convention that $\mathbf{0}$ denotes the unique empty vector over $\mathbb{F}_2$.
Concretely, \eqref{eq:pi_is_correct} means that it is always possible to transform an $n$-vertex tournament into any other using $(=p)$-inversions.
To prove \eqref{eq:pi_is_correct_linear_algebra}, it is enough to show that $\dim (\spn \{\mathbf{i}_X \mid X \subseteq [n], |X|=p\} )= \binom{n}{2}$.
For every $a,b \in [n]$ with $a<b$,
let $\mathbf{v}_{ab} \in \mathbb{F}_2^{\binom{n}{2}}$
be defined by, for every $i,j \in [n]$ with $i<j$,
\[
    (\mathbf{v}_{ab})_{\{i,j\}} = 
    \begin{cases}
        1 &\textrm{if $(a,b)=(i,j)$,} \\
        0 & \textrm{otherwise.}
    \end{cases}
\]
Clearly, the $\mathbf{v}_{ab}$ for $a,b \in [n]$ with $a<b$ form a family of linearly independent vectors of $\mathbb{F}_2^{\binom{n}{2}}$.
We will show that they all belong to $\spn \{\mathbf{i}_X \mid X \subseteq [n], |X|=p\}$.
Let $a,b \in [n]$ with $a<b$
and let $U \subseteq [n]$ be a $(p+2)$-set which contains $a,b$.
We claim that 
\[
    \mathbf{v}_{ab} = \sum_{X \subseteq U;\, |X|=p;\, a,b \in X} \mathbf{i}_X.
\]
Indeed, for every $i,j \in U$ with $i<j$,
\[
    |\{X \subseteq U \mid |X|=p; a,b,i,j \in X\}|
    =
    \begin{cases}
        \binom{p+2-4}{p-4} & \textrm{if $i,j,a,b$ are pairwise distinct,} \\
        \binom{p+2-3}{p-3} & \textrm{if $|\{i,j\} \cap \{a,b\}| = 1$,} \\
        \binom{p+2-2}{p-2} & \textrm{if $\{i,j\} = \{a,b\}$,}
    \end{cases}    
\]
which is odd if and only if $\{i,j\} = \{a,b\}$ since $p =2 \mod 4$.
This proves that $\mathbf{v}_{ab} = \sum_{X \subseteq U;\, |X|=p;\, a,b \in X} \mathbf{i}_X$,
and so \eqref{eq:pi_is_correct_linear_algebra} holds.
This proves the cases $p = 2 \mod 4$ in \Cref{thm:characterization_of_=p_equivalent_tournaments}
and in \Cref{thm:characterization_of_=p_invertible_tournaments}.

\subsubsection{Case \texorpdfstring{$p = 0 \mod 4$}{p = 0 mod 4}.}

Let $n$ be a positive integer, and let $p$ be a positive integer such that $n \geq p+2$
and $p = 0 \mod 4$.
Let $t=1$, and let $\pi$ be defined as follows.
For every $\mathbf{u} \in \mathbb{F}_2^{\binom{n}{2}}$, let
\[
    \pi(\mathbf{u}) = \sum_{i,j \in [n], i<j} \mathbf{u}_{\{i,j\}}.
\]

Less formally, $\pi$ is defined so that for any $T$, $\pi(\mathbf{a}_T)$ is the parity of the number of backward arcs (since the operations are in $\mathbb{F}_2$).

First, since $p = 0 \mod 4$, $\binom{p}{2}$ is even, and so $\pi(\mathbf{i}_X) = 0$ for every $(=p)$-set $X \subseteq [n]$, meaning $\spn \{\mathbf{i}_X \mid X \subseteq [n], |X|=p\} \subseteq \ker(\pi)$.
Hence, since $\pi$ is clearly surjective, and by the rank theorem, to prove \eqref{eq:pi_is_correct_linear_algebra}, it is now enough to show that $\dim (\spn \{\mathbf{i}_X \mid X \subseteq [n], |X|=p\}) \geq \binom{n}{2}-1$.
For every $a,b \in [n]$ with $a,b$ and $(a,b) \neq (1,2)$, let $\mathbf{w}_{ab} \in \mathbb{F}_2^n$ be the vector defined by

\[
(\mathbf{w}_{ab})_{\{i,j\}} = 
\begin{cases}
1 &\textrm{if $(i,j)=(a,b)$ or $(i,j)=(1,2)$,} \\
0 &\textrm{otherwise,}
\end{cases}
\]
for every $i,j \in [n]$ with $i<j$.
By Propositions~\ref{claim:2adjacentedges} and~\ref{claim:2edges},
$\mathbf{w}_{ab} \in \spn\{\mathbf{i}_X \mid X \subseteq [n], |X|=p\}$.
For every $(\alpha_{ab})_{a,b \in [n],\, a<b,\, (a,b) \neq (1,2)} \in \mathbb{F}_2^{\binom{n}{2}-1}$,
if
\[
\sum_{a,b \in [n], a<b, (a,b) \neq (1,2)} \alpha_{ab} \mathbf{w}_{ab} = \mathbf{0},
\]
then for every $a,b \in [n], a<b, (a,b) \neq (1,2)$, the coordinate corresponding to $ab$ gives $\alpha_{ab} = 0$.
It follows that $\alpha_{ab} = 0$ for every $a,b \in [n]$ with $a<b$ and $(a,b) \neq (1,2)$.
Hence the vectors $(\mathbf{w}_{ab})_{a,b \in [n], a<b, (a,b) \neq (1,2)}$ are linearly independent and belong to $\spn\{\mathbf{i}_X \mid X \subseteq [n], |X|=p\}$.
This proves $\dim (\spn \{\mathbf{i}_X \mid X \subseteq [n], |X|=p\} )\geq \binom{n}{2}-1$ and so \eqref{eq:pi_is_correct_linear_algebra},
and concludes the proof of the case $p= 0 \mod 4$ in \Cref{thm:characterization_of_=p_equivalent_tournaments}.

\begin{theorem}[Case $p = 0 \mod 4$ in \Cref{thm:characterization_of_=p_invertible_tournaments}]\label{cor:tournament-0mod4}
Let $p$ be an integer such that $p = 0 \mod 4$. Every tournament of order at least $p+2$ is $(=p)$-invertible. 
\end{theorem}

\begin{proof}
Let $n$ be an integer such that $n \geq p+2$.
It is enough to show that there are two acyclic tournaments $T_0,T_1$ on $[n]$ such that
$\pi(\mathbf{a}_{T_\epsilon}) = \epsilon$ for each $\epsilon \in \mathbb{F}_2$.
We take $A(T_0) = \{ij \mid 1\leq i<j \leq n\}$
and $A(T_1) = \{ij \mid 1\leq i<j \leq n\} \setminus \{12\} \cup \{21\}$.
\end{proof}

\subsubsection{Case \texorpdfstring{$p = 3 \mod 4$}{p = 3 mod 4}.}
Let $n$ be a positive integer, and let $p$ be a positive integer such that $n \geq p+2$
and $p = 3 \mod 4$.
Let $\pi$ be defined as follows.
For every $\mathbf{u} \in \mathbb{F}_2^{\binom{n}{2}}$, let $\pi(\mathbf{u})$ be the vector of $\mathbb{F}_2^{n-1}$ defined as:
\[
    \pi(\mathbf{u}) = \left(\textstyle\sum_{j \in [n] \setminus \{i\}} \mathbf{u}_{\{i,j\}}\right)_{i \in [n-1]}.
\]

Less formally, $\pi$ is such that for any tournament $T$, $\pi(\mathbf{a}_T)_i$ is the parity of the number of backward arcs incident to $i$ (as tail or head).
Hence 
\begin{eqnarray*}
    \pi(\mathbf{a}_T)_i & = & |N^+_T(i) \cap [i-1]| + |N^-_T(i) \cap \{i+1, \dots , n\}| \mod 2 \\
    & = & |N^+_T(i) \cap [i-1]| + n -i - |N^+_T(i) \cap \{i+1, \dots , n\}| \mod 2 \\
    & = & n+ i + d^+_T(i) \mod 2
\end{eqnarray*}
for every $i \in [n-1]$, which is indeed invariant under 
$(=p)$-inversions since $p$ is odd.
From the algebraic point of view, this means that for every $(=p)$-set $X \subseteq [n]$,
for every $i \in [n-1]$, $\sum_{j \in [n] \setminus \{i\}} (\mathbf{i}_{X})_{\{i,j\}} = 0$.
Therefore, $\pi(\mathbf{i}_X)=\mathbf{0}$ for every $X \subseteq [n]$ with $|X|=p$.
Thus, by the rank theorem, to prove \eqref{eq:pi_is_correct_linear_algebra}, it is now enough to show that
\begin{enumerate}
    \item $\pi$ is surjective, and
    \item $\dim (\spn \{\mathbf{i}_X \mid X \subseteq [n], |X|=p\}) \geq \binom{n}{2}-n+1 = \binom{n-1}{2}$.
\end{enumerate}

First we show that $\pi$ is surjective.
For every $a \in [n] \setminus \{1\}$,
let $\mathbf{y}_{a} \in \mathbb{F}_2^{\binom{n}{2}}$ be defined by
\[
(\mathbf{y}_{a})_{\{i,j\}}
=
\begin{cases}
1 & \textrm{if $(i,j)=(1,a)$,} \\
0 & \textrm{otherwise,}
\end{cases}
\]
for every $i,j \in [n]$ with $i<j$.
Then the matrix whose columns are $\pi(\mathbf{y}_a)$ for $a \in [n] \setminus \{1\}$
is, up to a permutation of the columns, the identity matrix.
We conclude that $\pi$ is surjective.

We now show that $\dim (\spn \{\mathbf{i}_X \mid X \subseteq [n], |X|=p\} )\geq \binom{n-1}{2}$.
Let $a,b \in [n] \setminus \{1\}$ with $a<b$.
Let $\mathbf{z}_{ab} \in \mathbb{F}_2^{\binom{n}{2}}$ be defined
by
\[
(\mathbf{z}_{ab})_{\{i,j\}} =
\begin{cases}
    1&\textrm{if $(i,j) \in \{(1,2), (2,a), (a,b), (b,1)\}$,} \\
    0 &\textrm{otherwise,}
\end{cases}
\]
for every $i,j \in [n]$ with $i<j$.
First, if $a>2$, then $\{a,b\} \cap \{1,2\} = \emptyset$, and so
$\mathbf{z}_{ab} \in \spn \{\mathbf{i}_X \mid X \subseteq [n], |X|=p\}$ by Proposition~\ref{claim:4cycle}.
Suppose now that $a = 2$.
Let $U \subseteq [n]$ be a set of size $p+2$ containing $1,2,b$. Observe that such a set exists as $n \geq p+2$.
We claim that 
\[
    \mathbf{z}_{ab} = \sum_{X \subseteq U; |X|=p; 1,2,b \in X} \mathbf{i}_X.
\]
Indeed, for every $i,j \in U$ with $i<j$,
\[
    |\{X \subseteq U \mid |X|=p; 1,2,b,i,j \in X\}|
    =
    \begin{cases}
        \binom{p+2-5}{p-5} & \textrm{if $\{i,j\} \cap \{1,2,b\} = \emptyset$,} \\
        \binom{p+2-4}{p-4} & \textrm{if $|\{i,j\} \cap \{1,2,b\}| = 1$,} \\
        \binom{p+2-3}{p-3} & \textrm{if $\{i,j\} \subseteq \{1,2,b\}$,}
    \end{cases}    
\]
which is odd if and only if $\{i,j\} \subseteq \{1,2,b\}$ since $p = 3 \mod 4$.
It remains to show that $\mathbf{z}_{ab}$ for $a,b \in [n] \setminus \{1\}$ with $a<b$
are linearly independent.
Suppose that $\alpha_{ab} \in \mathbb{F}_2$ for $a,b \in [n] \setminus \{1\}, a<b$
are such that
\[
    \sum_{a,b \in [n] \setminus \{1\}, a<b} \alpha_{ab} \mathbf{z}_{ab} = \mathbf{0}.
\]
First, for every $a,b \in [n] \setminus \{1,2\}$ with $a<b$,
the coordinate corresponding to $\{a,b\}$ implies that $\alpha_{ab} = 0$.
Hence, $\sum_{b \in [n] \setminus \{1,2\}} \alpha_{2b} \mathbf{z}_{2b} = \mathbf{0}$.
Now, for every $b \in [n] \setminus \{1,2\}$,
the coordinate corresponding to $\{2,b\}$ implies that $\alpha_{2b} = 0$.
This proves that the vectors $\mathbf{z}_{ab}$ for $a,b \in [n] \setminus \{1\}$ with $a<b$
are linearly independent,
and concludes the proof of \eqref{eq:pi_is_correct_linear_algebra}.
This shows the case $p = 3 \mod 4$ in \Cref{thm:characterization_of_=p_equivalent_tournaments}

\medskip

A vertex in a digraph $D$ is {\bf out-even} (resp. {\bf out-odd}) if its out-degree in $D$ is even (resp. odd).

\begin{theorem}[Case $p = 3 \mod 4$ in \Cref{thm:characterization_of_=p_invertible_tournaments}]\label{cor:p-odd-3mod4}
    Let $p$ be a positive integer with $p = 3 \mod 4$.
    For every tournament $T$ on $n \geq p+2$ vertices,
    $T$ is $(=p)$-invertible if and only if it has $\lceil n/2 \rceil$ out-even vertices.
\end{theorem}

\begin{proof}
    For every $\mathbf{u} \in \mathbb{F}_2^{\binom{n}{2}}$,
    let $\pi'(\mathbf{u}) \in \mathbb{F}_2^n$ be defined by
    \begin{align*}
        \pi'(\mathbf{u})_i &= 
        \pi(\mathbf{u})_i + n + i \mod 2,
    \intertext{for every $i \in [n-1]$, and}
        \pi'(\mathbf{u})_{n} &= \binom{n}{2} + \sum_{i\in[n-1]} \pi(\mathbf{u})_i \mod 2.
    \end{align*}
    Observe that for every tournament $T$ on $[n]$,
    for every $i \in [n]$, $\pi'(\mathbf{a}_T)_i = d^+_T(i) \mod 2$. 
    Indeed, for every $i \in [n-1]$,
    \begin{align*}
        d^+_T(i) 
        &= |\{j \in [i-1] \mid ij \in A(T)\}| + |\{j \in \{i+1, \dots, n\} \mid ij \in A(T)\}| \\
        &= |\{j \in [i-1] \mid ij \in A(T)\}| + (n-i) - |\{j \in \{i+1, \dots, n\} \mid ji \in A(T)\}| \\
        &= n-1 + |\{j \in [n] \setminus \{i\} \mid \text{$ij \in A(T)$ and $j<i$, or $ji \in A(T)$ and $j>i$}\}| \\
        &= n+i + \pi(\mathbf{a}_T)_i \mod 2.
    \end{align*}
    For $i=n$, the equality $\pi'(\mathbf{a}_T)_i = d^+_T(i) \mod 2$ then follows
    from the fact that $\sum_{i \in [n]} |\{j \in [n] \setminus \{i\} \mid \text{$ij \in A(T)$ and $j<i$, or $j>i$ and $ji \in A(T)$}\}| = \binom{n}{2}$
    and $d^+_T(n) = |\{j \in [n-1] \mid \text{$ij \in A(T)$ and $j<i$, or $ji \in A(T)$ and $j>i$}\}|$.
    This proves that, for every $i \in [n]$,
    $\pi'(\mathbf{a}_T)_i = d^+_T(i) \mod 2$, and so $\pi'(\mathbf{a}_T)_i=0$ if vertex $i$ is out-even and $\pi'(\mathbf{a}_T)_i=1$ if vertex $i$ is out-odd.
    Moreover, for every $\mathbf{u},\mathbf{u}' \in \mathbb{F}_2^{\binom{n}{2}}$,
    $\pi'(\mathbf{u}) = \pi'(\mathbf{u}')$ if and only if $\pi(\mathbf{u}) = \pi(\mathbf{u}')$.
    Hence, it remains to show that for every $\mathbf{v} \in \mathbb{F}_2^{n-1}$,
    there exists an acyclic tournament $T$ on $[n]$ such that $\pi'(\mathbf{a}_T) = \mathbf{v}$
    if and only if the number of zeros in $\mathbf{v}$ is $\lceil n/2 \rceil$.

    First, for every transitive tournament $T$,
    the number of vertices of even out-degree is $\lceil n/2 \rceil$,
    and so the number of zeros in $\mathbf{a}_T$ is $\lceil n/2 \rceil$.
    Conversely,
    for every $\mathbf{v} \in \mathbb{F}_2^{n-1}$ with $\lceil n/2 \rceil$ zeros,
    let $\sigma$ be a permutation of $[n]$
    such that
    for every $i \in [n]$,
    $\mathbf{v}_{\sigma(i)} = 1$ if and only if $n+i$ is odd.
    Then, $T = ([n], \{\sigma(i)\sigma(j) \mid 1 \leq i < j \leq n\})$ is an acyclic tournament
    with $\pi'(\mathbf{a}_T) = \mathbf{v}$.
    This proves the corollary.
\end{proof}

\subsubsection{Case \texorpdfstring{$p = 1 \mod 4$}{p = 1 mod 4}.}

Let $n$ be a positive integer, and let $p$ be a positive integer such that $n \geq p+2$
and $p = 1 \mod 4$.
We define $\pi$ as follows.
For every $\mathbf{u} \in \mathbb{F}_2^{\binom{n}{2}}$, let
\[
    \pi(\mathbf{u}) = \left(\left(\textstyle\sum_{j \in [n] \setminus \{i\}} \mathbf{u}_{\{i,j\}}\right)_{i \in [n-1]},\textstyle\sum_{i,j \in [n], i<j} \mathbf{u}_{\{i,j\}}\right).
\]

Since $p = 1 \mod 4$, $\binom{p}{2}$ is even and so $\sum_{i,j \in [n], i<j} (\mathbf{i}_X)_{\{i,j\}} = 0$ for every $X \subseteq [n]$ with $|X|=p$.
Moreover, since $p$ is odd, for every $i \in [n]$, $\sum_{j \in [n] \setminus \{i\}} (\mathbf{i}_{X})_{\{i,j\}} = 0$.
Therefore, $\pi(\mathbf{i}_X)=\mathbf{0}$ for every $(=p)$-set $X \subseteq [n]$.
Thus, to prove \eqref{eq:pi_is_correct_linear_algebra}, by the rank theorem, it is now enough to show that
\begin{enumerate}
    \item $\pi$ is surjective, and
    \item $\dim (\spn \{\mathbf{i}_X \mid X \subseteq [n], |X|=p\}) \geq \binom{n}{2}-n$.
\end{enumerate}

First we show that $\pi$ is surjective.
For every $a \in [n-1]$, let $\mathbf{t}_a \in \mathbb{F}_2^{\binom{n}{2}}$ be defined by
\[
    (\mathbf{t}_a)_{\{i,j\}}=
    \begin{cases}
        1 &\textrm{if $(i,j) = (a,n)$,} \\
        0 &\textrm{otherwise,}
    \end{cases}
\]
for every $i,j \in [n]$ with $i<j$.
Moreover, let $\mathbf{t}_n \in \mathbb{F}_2^{\binom{n}{2}}$ be defined by
\[
    (\mathbf{t}_n)_{\{i,j\}}=
    \begin{cases}
        1 &\textrm{if $(i,j) \in \{(1,2),(1,3),(2,3)\}$} \\
        0 &\textrm{otherwise,}
    \end{cases}
\]
for every $i,j \in [n]$ with $i<j$.
The matrix whose columns are $(\pi(\mathbf{t}_a))_{a \in [n]}$ is, up to the ordering of the columns,
\[
    \begin{pmatrix}
        1      & 0 & \dots  & 0 & 0      \\
        0      & 1 &        & 0 & 0      \\
        \vdots &   & \ddots &   & \vdots \\
        0      & 0 & \dots  & 1 & 0      \\
        1      & 1 & \dots  & 1 & 1      \\
    \end{pmatrix}
\]
and so $\pi$ is surjective.

\medskip

We now show that $\dim (\spn \{\mathbf{i}_X \mid X \subseteq [n], |X|=p\})\geq \binom{n}{2}-n$.
For every $a,b \in [n] \setminus \{1,2\}$ with $a<b$, let $\mathbf{s}_{ab} \in \mathbb{F}_2^{\binom{n}{2}}$ be defined by
\[
    (\mathbf{s}_{ab})_{\{i,j\}} =
    \begin{cases}
        1 &\textrm{if $(i,j) \in \{(1,2), (2,b), (a,b), (1,b)\}$}, \\
        0 & \textrm{otherwise,}
    \end{cases}
\]
for every $i,j \in [n]$ with $i<j$.
By Proposition~\ref{claim:4cycle}, $\mathbf{s}_{ab} \in \spn\{\mathbf{i}_X \mid X \subseteq [n], |X|=p\}$.
Moreover, for every $a\in \{3, \dots, n-1\}$, let $\mathbf{s}_a \in \mathbb{F}_2^{\binom{n}{2}}$ be defined by
\[
    (\mathbf{s}_a)_{\{i,j\}} =
    \begin{cases}
        1 &\textrm{if $(i,j) \in \{(1,2), (2,a), (a,a+1), (1, a+1)\}$,} \\
        0 &\textrm{otherwise,} 
    \end{cases}
\]
for every $i,j \in [n]$ with $i<j$.
Again, by Proposition~\ref{claim:4cycle}, $\mathbf{s}_{a} \in \spn\{\mathbf{i}_X \mid X \subseteq [n], |X|=p\}$.
We claim that the vectors $(\mathbf{s}_{ab})_{a,b \in [n] \setminus \{1,2\},a<b}, (\mathbf{s}_a)_{a \in \{3, \dots, n-1\}}$ are linearly independent.
Suppose for contradiction that there exists $\alpha_{ab} \in \mathbb{F}_2$ for $a,b \in [n] \setminus \{1,2\}, a<b$, and $\alpha_a \in \mathbb{F}_2$ for $a \in \{3, \dots, n-1\}$
not all zero such that
\[
    \sum_{a,b \in [n] \setminus \{1,2\}, a<b} \alpha_{ab} \mathbf{s}_{ab} + \sum_{a \in \{3, \dots, n-1\}} \alpha_a \mathbf{s}_a = \mathbf{0}.
\]

First assume that $\alpha_a = 0$ for every $a \in \{3, \dots, n-1\}$.
Then for every $a,b \in [n] \setminus \{1,2\}$ with $a<b$,
the coordinate corresponding to $\{a,b\}$ implies that $\alpha_{ab} = 0$.
This is a contradiction.

Now suppose that there exists $a_0 \in \{3, \dots, n-1\}$ such that $\alpha_{a_0} = 1$.
Take such an $a_0$ maximum.
The coordinate corresponding to $\{a_0,a_0+1\}$ implies that $\alpha_{(a_0,a_0+1)} = 1$.
But then, the coordinate corresponding to $\{1,a_0+1\}$ implies that there exists $b \in [n] \setminus \{1,2\}$ with $b>a_0+1$ such that
$\alpha_{(a_0+1,b)} = 1$.
As a consequence, the coordinate corresponding to $\{a_0+1,b\}$ implies that
$\alpha_{b-1} = 1$, contradicting the maximality of $a_0$.
Therefore, the vectors $(\mathbf{s}_{ab})_{a,b \in [n] \setminus \{1,2\},a<b}, (\mathbf{s}_a)_{a \in \{3, \dots, n-1\}}$ are linearly independent
and so 
\[
    \dim (\spn \{\mathbf{i}_X \mid X \subseteq [n], |X|=p\}) \geq \binom{n-2}{2} + n-3 = \binom{n}{2}-n.
\]
This proves \eqref{eq:pi_is_correct_linear_algebra},
and shows the case $p = 1 \mod 4$ in \Cref{thm:characterization_of_=p_equivalent_tournaments}.

\begin{theorem}[Case $p = 1 \mod 4$ in \Cref{thm:characterization_of_=p_invertible_tournaments}]\label{cor:p-odd-1mod4}
    Let $p$ be a positive integer with $p = 1 \mod 4$.
    For every tournament $T$ on $n \geq p+2$ vertices,
    $T$ is $(=p)$-invertible if and only if it has $\lceil n/2 \rceil$ out-even vertices.
\end{theorem}

\begin{proof}
    For every $\mathbf{u} \in \mathbb{F}_2^{\binom{n}{2}}$,
    let $\pi'(\mathbf{u}) \in \mathbb{F}_2^{n+1}$ be defined by
    \[
        \pi'(\mathbf{u})_i = 
        \pi(\mathbf{u})_i + n + i \mod 2,
    \]
    for every $i \in [n-1]$,
    \[
        \pi'(\mathbf{u})_{n} = \binom{n}{2} + \sum_{i \in [n-1]} \pi'(u)_i \mod 2,
    \]
    and
    \[
        \pi'(\mathbf{u})_{n+1} = \pi(\mathbf{u})_n.
    \]
    Observe that for every tournament $T$ on the vertex set $[n]$,
    for every $i \in [n]$,
    $\pi'(\mathbf{a}_T)_i = d^+_T(i) \mod 2$,
    and $\pi'(\mathbf{a}_T)_{n+1}$ is the parity of the number of backward arcs. 
    Moreover, for every $\mathbf{u},\mathbf{u}' \in \mathbb{F}_2^{\binom{n}{2}}$,
    $\pi'(\mathbf{u}) = \pi'(\mathbf{u}')$ if and only if $\pi(\mathbf{u}) = \pi(\mathbf{u}')$.
    Hence, it remains to show that for every $\mathbf{v} \in \mathbb{F}_2^{n-1}$,
    there exists an acyclic tournament $T$ on $[n]$ such that $\pi'(\mathbf{a}_T) = \mathbf{v}$
    if and only if the number of zeros in $\{\mathbf{v}_i \mid i \in [n]\}$ is $\lceil n/2 \rceil$.

    First, for every transitive tournament $T$ on $[n]$,
    the number of out-even vertices is $\lceil n/2 \rceil$,
    and so the number of zeros in $\pi'(\mathbf{a}_T)$ is $\lceil n/2 \rceil$.
    Reciprocally,
    for every $\mathbf{v} \in \mathbb{F}_2^{n+1}$ with exactly $\lceil n/2 \rceil$ zeros in $(\mathbf{v}_1, \dots, \mathbf{v}_n)$,
    let $\sigma$ be a permutation of $[n]$
    such that
    for every $i \in [n]$,
    $\mathbf{v}_{\sigma(i)} = 1$ if and only if $n+i$ is odd.
    By possibly exchanging the position of $1$ and $3$ in $\sigma$,
    we can assume that $|\{ij \mid 1 \leq i < j \leq n, \sigma(i) > \sigma(j)\}| \mod 2 = \mathbf{v}_{n+1}$.
    Then, $T = ([n], \{\sigma(i)\sigma(j) \mid i<j\})$ is an acyclic tournament
    with $\pi'(\mathbf{a}_T) = \mathbf{v}$.
    This proves the corollary.
\end{proof}

\subsection{Polynomial-time algorithm for general oriented graphs}\label{subsec:algo-gen}

The purpose of this section is to extend the algorithmic result proved in \Cref{subsec:carac} for tournaments to arbitrary oriented graphs. More precisely, we prove Theorems~\ref{algo} and \ref{algoeven}, which we restate.

\algo*

\algoeven*

The decisive ingredient is the following result, which gives the connection between invertible tournaments and arbitrary invertible oriented graphs.
\begin{lemma}\label{lem:=pinvert-reduc}
    An oriented graph $D$ is $(=p)$-invertible if and only if there is an $(=p)$-invertible tournament $T$ having $D$ as spanning oriented subgraph. 
\end{lemma}

\begin{proof}
    If a tournament $T$ admits $D$ as a spanning oriented subgraph, any decycling $(=p)$-family of $T$ is also a decycling $(=p)$-family of $D$.
    Conversely, assume $D$ admits a decycling $(=p)$-family $\cal X$. 
    Let $D_1=\Inv(D; {\cal X})$. Since it is acyclic, $D_1$ admits an acyclic ordering $\sigma$. Let $T_1$ be the transitive tournament on $V(D)$ with acyclic ordering $\sigma$ and let $T= \Inv(T_1; {\cal X})$.
    Now observe that $D$ is a spanning subgraph of $T$. Moreover $\Inv(T; {\cal X}) = T_1$, so $\cal X$ is a decycling $(=p)$-family of $T$.
\end{proof}

We now give the simple proof of \Cref{algoeven}.
\begin{proof}
    Let $T$ be a tournament on $V(D)$ such that all edges of $UG(D)$ have the same orientation in $D$ and $T$. 
    It now follows from \Cref{thm:characterization_of_=p_invertible_tournaments} that $T$ is $(=p)$-invertible. Moreover, as $T$ contains $D$ as an oriented subgraph, we obtain by \Cref{lem:=pinvert-reduc} that $D$ is $(=p)$-invertible. 
\end{proof}

The proof of \Cref{algo} is more involved because when $p$ is odd, there are tournaments which are not $(=p)$-invertible. The idea is thus to characterize when an oriented graph can be extended by adding arcs into an $(=p)$-invertible tournament.  \Cref{lem:equiorient} provides such a characterization in terms of the existence of an orientation of the complement of the underlying graph with some degree constraint. We then show in \Cref{lem:findorient} how to decide whether such an orientation exists in polynomial time.

Given an oriented graph $\ora{G}$ and a bipartition $(V_1,V_2)$ of $V(\ora{G})$, 
we denote by $\gamma_{\ora{G}}(V_1,V_2)$ the integer $|\{v \in V_1 \mid d_{\ora{G}}^+(v)\text{ is even}\}| + |\{v \in V_2 \mid d_{\ora{G}}^+(v)\text{ is odd}\}|$. 

\begin{lemma}\label{lem:equiorient}
    Let $p\geq 3$ be an odd integer, 
    $D$ be an oriented graph, 
    $G$ be the complement of the underlying undirected graph $UG(D)$ of $D$,
    $V_1=\{v \in V(D) \mid d_D^+(v)\text{ is even}\}$ and $V_2=V(D)\setminus V_1$. 
    Then $D$ is $(=p)$-invertible if and only if there exists an orientation $\vec{G}$ of $G$ with $\gamma_{\vec{G}}(V_1,V_2)=\lceil \frac{n}{2}\rceil$. 
\end{lemma}
\begin{proof}
    First suppose that $D$ is $(=p)$-invertible, so by Lemma \ref{lem:=pinvert-reduc}, there exists an $(=p)$-invertible tournament $T$ that contains $D$ as an oriented subgraph. 
    It follows from \Cref{cor:p-odd-3mod4} and \Cref{cor:p-odd-1mod4} that $T$ contains exactly $\lceil\frac{n}{2}\rceil$ out-even vertices. Let $\vec{G}$ be the orientation of $G$ inherited from $T$.
    For every $v \in V(D)$, as $d_T^+(v)=d_{\vec{G}}^+(v)+d_D^+(v)$,
    we obtain that $v$ contributes to $\gamma_{\vec{G}}(V_1,V_2)$ (i.e. $v \in V_1$ and $d^+_{\vec{G}}(v)$ even, or $v \in V_2$ and $d^+_{\vec{G}}(v)$ odd)
    if and only if it is an out-even vertex in $T$. 
    It follows that $\gamma_{\vec{G}}(V_1,V_2)=\lceil\frac{n}{2}\rceil$.
    
    Now suppose that there exists an orientation $\vec{G}$ of $G$ with  $\gamma_{\vec{G}}(V_1,V_2)=\lceil \frac{n}{2}\rceil$. 
    Let $T$ be the tournament on $V(D)$ in which every edge in $UG(D)$ has the same orientation as in $D$ and every edge in $E(G)$ has the same orientation as in $\vec{G}$.
    For every $v \in V(D)$, as $d_T^+(v)=d_{\vec{G}}^+(v)+d_D^+(v)$, we obtain that $v$ contributes to $\gamma_{\vec{G}}(V_1,V_2)$ if and only if it is an out-even vertex in $T$. 
    It follows that $T$ contains exactly $\lceil \frac{n}{2}\rceil$ out-even vertices. 
    As a consequence, by \Cref{cor:p-odd-3mod4} and \Cref{cor:p-odd-1mod4}, we obtain that $T$ is $(=p)$-invertible. It follows by Lemma \ref{lem:=pinvert-reduc} that $D$ is $(=p)$-invertible.
\end{proof}

In the following, we show that it is possible to efficiently decide whether a given graph has an orientation attaining a given value for $\gamma$. The following simple observation will be used several times.
\begin{proposition}\label{prop:parite}
    Let $\ora{G}$ be an oriented graph and $(V_1,V_2)$ a bipartition of $V(\ora{G})$. Then $\gamma_{\ora{G}}(V_1,V_2)=| V_1|+|A(\ora{G})| \mod 2$.
\end{proposition}
\begin{proof}
    We have the following equalities modulo $2$:
    \begin{align*}
        \gamma_{\ora{G}}(V_1,V_2)&=|\{v \in V_1 \mid d_{\ora{G}}^+(v)\text{ is even}\}| + |\{v \in V_2 \mid d_{\ora{G}}^+(v)\text{ is odd}\}|\\
        &=2|\{v \in V_1 \mid d_{\ora{G}}^+(v)\text{ is odd}\}|+|\{v \in V_1 \mid d_{\ora{G}}^+(v)\text{ is even}\}| \\
        &\hspace{1cm}+ |\{v \in V_2 \mid d_{\ora{G}}^+(v)\text{ is odd}\}|\\
        &=|V_1|+ |\{v \in V({\ora{G}}) \mid d_{\ora{G}}^+(v)\text{ is odd}\}|\\
        &=|V_1|+|A({\ora{G}})| \mod{2}. \qedhere
     \end{align*}
     
\end{proof}
We are now ready to give a characterization for the case where the graph is connected.
\begin{proposition}\label{prop:findorientconnect}
   Let $G$ be a connected graph, $(V_1,V_2)$ a bipartition  of $V(G)$ and $s$ an integer. Then  
     there exists an orientation $\vec{G}$ of $G$ with $\gamma_{\vec{G}}(V_1,V_2)=s$ if and only if $s=|V_1|+|E(G)| \mod{2}$ and $0\leq s \leq |V(G)|$ hold. 
\end{proposition}
\begin{proof}
    First suppose that there exists an orientation $\vec{G}$ of $G$ with $\gamma_{\vec{G}}(V_1,V_2)=s$. 
    As $0\leq \gamma_{\vec{G}}(V_1,V_2)$ and $\gamma_{\vec{G}}(V_1,V_2)\leq |V(\vec{G})|$ hold by definition, we obtain $s \geq 0$ and $s \leq |V(\vec{G})|=|V(G)|$. 
    Further, it follows from Proposition \ref{prop:parite} that $s=\gamma_{\vec{G}}(V_1,V_2)=|V_1|+|E(G)|=|A(D)| \mod{2}$.

    Now suppose that $s=|V_1|+|E(G)| \mod{2}$ and $0\leq s \leq |V(G)|$. Then there exists a set $X \subseteq V(G)$ with $|X|=s$. In the following, for an orientation $\vec{G}$ of $G$ and some $v \in V(G)$, we say that $v$ is {\it good} in $\vec{G}$ if one of the following holds:
    \begin{itemize}
        \item $v \in (V_1\cap X) \cup (V_2 \setminus X)$ and $d_{\vec{G}}^+(v)$ is even,
         \item $v \in (V_2\cap X) \cup (V_1 \setminus X)$ and $d_{\vec{G}}^+(v)$ is odd.
    \end{itemize}

    Now let $\vec{G}_0$ be an orientation of $G$ which maximizes the number of vertices in $V(G)$ which are good in $\vec{G}_0$. 
    Let $S$ be the set of good vertices in $\vec{G}_0$. 
    As $|X|=s=|V_1|+|E(G)| \mod{2}$, we have (modulo $2$):    
    \begin{align*}
        |S|
        &= |\{v \in (V_1\cap X) \cup (V_2 \setminus X) \mid d_{\vec{G}_0}^+(v)\text{ is even }\}| \\
        &\hspace{1cm}+|\{v \in (V_2\cap X) \cup (V_1 \setminus X) \mid d_{\vec{G}_0}^+(v)\text{ is odd }\}|\\
        &= |\{v \in (V_1\cap X) \cup (V_2 \setminus X) \mid d_{\vec{G}_0}^+(v)\text{ is even }\}| \\
        &\hspace{1cm}+|\{v \in (V_1\cap X) \cup (V_2 \setminus X) \mid d_{\vec{G}_0}^+(v)\text{ is odd }\}|\\
        &\hspace{1cm}+|\{v \in (V_1\cap X) \cup (V_2 \setminus X) \mid d_{\vec{G}_0}^+(v)\text{ is odd }\}| \\
        &\hspace{1cm}+|\{v \in (V_2\cap X) \cup (V_1 \setminus X) \mid d_{\vec{G}_0}^+(v)\text{ is odd }\}|\\
        &=|V_1\cap X|+|V_2 \setminus X|+|E(G)|\\
        &=|V_1\cap X|+|V_2 \setminus X|+|X|+|V_1|\\
        &=|V_1\cap X|+|V_2\cap X|+|V_2 \setminus X|+|V_2\cap X|+|X|+|V_1|\\
        &=|X|+|V_2|+|X|+|V_1|\\
        &=|V(G)|.
    \end{align*}

    First suppose that $S\neq V(G)$. Then, as $|S|=|V(G)|\mod{2}$, we obtain that $|S|\leq |V(G)|-2$, so there exist distinct vertices $u,v\in V(G)\setminus S$. As $G$ is connected, there exists a $uv$-path $P$ in $G$. Let $\vec{G}_1$ be the orientation of $G$ which is obtained from $\vec{G}_0$ by reversing the orientation of all arcs in $E(P)$. For every $w \in S$, we have $d_{\vec{G}_1}^+(w)\in \{d_{\vec{G}_0}^+(w)-2,d_{\vec{G}_0}^+(w),d_{\vec{G}_0}^+(w)+2\}$, so $w$ is good in $\vec{G}_1$. Moreover, we have $d_{\vec{G}_1}^+(u)\in \{d_{\vec{G}_0}^+(u)-1,d_{\vec{G}_0}^+(u)+1\}$, so $u$ is good in $\vec{G}_1$. We obtain a contradiction to the maximality of $\vec{G}_0$. It follows that $S=V(G)$.

    This yields that every vertex in $V(G)$ contributes to $\gamma_{\vec{G}_0}(V_1,V_2)$ if and only if it is contained in $X$. It follows that $\gamma_{\vec{G}_0}(V_1,V_2)=|X|=s$.
\end{proof}
We are now ready to extend this result to arbitrary graphs.
\begin{lemma}\label{lem:findorient}
    Given a graph $G$, a bipartition $(V_1,V_2)$ of $V(G)$ and an integer $s$, 
    one can decide in polynomial time whether there exists an orientation $\vec{G}$ of $G$ with $\gamma_{\vec{G}}(V_1,V_2)=s$. 
\end{lemma}
\begin{proof}
Let $\mathcal{K}$ be the set of components of $G$. We now define two functions $\ell,u:\mathcal{K}\rightarrow \mathbb{Z}_{\geq 0}$.
First, for every $K \in \mathcal{K}$, we set $\ell(K)= |V_1\cap V(K)|+|E(K)|\mod{2}$.
Then we set $u(K)=|V(K)|$ if $|V_1\cap V(K)|+|E(K)|+|V(K)|=0 \mod{2}$ and $\ell(K)=|V(K)|-1$ otherwise.
We will show that there exists an orientation $\vec{G}$ with $\gamma_{\vec{G}}(V_1,V_2)=s$ of $G$ if and only if $s=|V_1|+|E(G)|\mod{2}$ and $\sum_{K \in \mathcal{K}}\ell(K)\leq s \leq \sum_{K \in \mathcal{K}}u(K)$ holds. This clearly yields the desired algorithm.

First suppose that there exists an orientation $\vec{G}$ of $G$ with $\gamma_{\vec{G}}(V_1,V_2)=s$.
The first condition is a direct application of Proposition~\ref{prop:parite}, yielding $s=|V_1|+|E(G)|\mod{2}$.
Now, for every $K \in \mathcal{K}$, let $\vec{K}$ be the orientation of $K$ that is inherited from $\vec{G}$. It follows by Proposition \ref{prop:findorientconnect} that $\gamma_{\vec{K}}(V_1\cap V(K),V_2 \cap V(K))\geq \ell(K)$ holds for every $K \in \mathcal{K}$. We obtain that $s=\gamma_{\vec{G}}(V_1,V_2)=\sum_{K \in \mathcal{K}}\gamma_{\vec{K}}(V_1\cap V(K),V_2 \cap V(K))\geq \sum_{K \in \mathcal{K}}\ell(K)$. Similarly, it follows by Proposition \ref{prop:findorientconnect} that $\gamma_{\vec{K}}(V_1\cap V(K),V_2 \cap V(K))\leq u(K)$ holds for every $K \in \mathcal{K}$.
We obtain that $s=\gamma_{\vec{G}}(V_1,V_2)=\sum_{K \in \mathcal{K}}\gamma_{\vec{K}}(V_1\cap V(K),V_2 \cap V(K))\leq \sum_{K \in \mathcal{K}}\ell(K)$. 

For the other direction, suppose $s=|V_1|+|E(G)|\mod{2}$ and $\sum_{K \in \mathcal{K}}\ell(K)\leq s \leq \sum_{K \in \mathcal{K}}u(K)$.
We will first need the following intermediate claim.
\begin{claim}
    There exists a function $t\colon \mathcal{K}\rightarrow \mathbb{Z}_{\geq 0}$ such that $\sum_{K \in \mathcal{K}}\ell(K)+2t(K)=s$ and for every $K \in \mathcal{K}$, we have $0 \leq t(K)\leq \frac{1}{2}(u(K)-\ell(K))$.
\end{claim}

\begin{proofclaim}
    Observe that for every $K \in \mathcal{K}$, we have $\ell(K)=|V_1\cap V(K)|+|E(K)|=u(K)\mod{2}$. In particular, as $\ell(K)\leq 1$ and $u(K)\geq 0$, we have $\ell(K)\leq u(K)$. Now let $t \colon \mathcal{K}\rightarrow \mathbb{Z}_{\geq 0}$ be a function that maximizes $\sum_{K \in \mathcal{K}}\ell(K)+2t(K)$ among all integer-valued functions such that $\sum_{K \in \mathcal{K}}\ell(K)+2t(K)\leq s$ and for every $K \in \mathcal{K}$, we have $0 \leq t(K)\leq \frac{1}{2}(u(K)-\ell(K))$. By the above remark, we have that the function mapping every component to 0 is such a function, so $t$ is well-defined. Suppose for the sake of a contradiction that $\sum_{K \in \mathcal{K}}\ell(K)+2t(K)<s$. By assumption, $s=|V_1|+|E(G)|=\sum_{K \in \mathcal{K}}|V_1 \cap V(K)|+|E(K)|=\sum_{K \in \mathcal{K}}\ell(K)=\sum_{K \in \mathcal{K}}\ell(K)+2t(K) \mod{2}$. Then, we obtain that $\sum_{K \in \mathcal{K}}\ell(K)+2t(K)\leq s-2$. Next, as $\sum_{K \in \mathcal{K}}\ell(K)+(u(K)-\ell(K))=\sum_{K \in \mathcal{K}}u(K)\geq s>\sum_{K \in \mathcal{K}}\ell(K)+2t(K)$, there exists some $K_0 \in \mathcal{K}$ such that $t(K_0)<\frac{1}{2}(u(K_0)-\ell(K_0))$. Next, as $\ell(K_0)=u(K_0)\mod{2}$, we obtain that $t(K_0)\leq \frac{1}{2}(u(K_0)-\ell(K_0))-1$. We now define $t'\colon \mathcal{K}\rightarrow \mathbb{Z}_{\geq 0}$ by $t'(K_0)=t(K_0)+1$ and $t'(K)=t(K)$ for all $K \in \mathcal{K}\setminus \{K_0\}$. As $t(K_0)\leq \frac{1}{2}(u(K_0)-\ell(K_0))-1$, we obtain that $0 \leq t'(K)\leq \frac{1}{2}(u(K)-\ell(K))$ holds for all $K \in \mathcal{K}$. Finally, as $\sum_{K \in \mathcal{K}}\ell(K)+2t(K)\leq s-2$, we obtain that $\sum_{K \in \mathcal{K}}\ell(K)+2t(K)<\sum_{K \in \mathcal{K}}\ell(K)+2t'(K)\leq s$. This contradicts the choice of $t$.
\end{proofclaim}

In the following, we take any $t$ given by the claim above.
For every $K\in \mathcal{K}$, we clearly have $\ell(K)+2t(K)\geq \ell(K)\geq 0$ and $\ell(K)+2t(K)\leq u(K)\leq |V(K)|$. Further, we have $\ell(K)+2t(K)=\ell(K)=|V_1\cap V(K)|+|E(K)|\mod{2}$. It follows by Proposition \ref{prop:findorientconnect} that there exists an orientation $\vec{K}$ of $K$ with $\gamma_{\vec{K}}(V_1\cap V(K),V_2 \cap V(K))=\ell(K)+2t(K)$. Now let $\vec{G}$ be the orientation of $G$ in which every $e \in E(K)$ has the orientation it has in $\vec{K}$ for every $e \in E(K)$ and every $K \in \mathcal{K}$. We obtain that $\gamma_{\vec{G}}(V_1,V_2)=\sum_{K \in \mathcal{K}}\gamma_{\vec{K}}(V_1\cap V(K),V_2 \cap V(K))=\sum_{K \in \mathcal{K}}\ell(K)+2t(K)=s$.
\end{proof}

We are now ready to give the proof of \Cref{algo}.
\begin{proof}[Proof of \Cref{algo}]
    If $p \geq n$, then $D$ is $(=p)$-invertible if and only if it is acyclic, which can be checked in polynomial time. Now suppose $p \leq n-2$.
    If $p$ is even, then the statement directly follows from Theorem~\ref{algoeven}.
    Now suppose that $p$ is odd. 
    Let $G$ be the complement of $UG(D)$, 
    $V_1 = \{v \in V(D) \mid d_D^+(v)\text{ is even}\}$ and $V_2 = V(D)\setminus V_1$. 
    Using Lemma~\ref{lem:findorient}, in polynomial time, one can decide whether $G$ admits an orientation $\vec{G}$ with $\gamma_{\vec{G}}(V_1,V_2)=\lceil\frac{n}{2}\rceil$. 
    If no such orientation, exists, we obtain from Lemma~\ref{lem:equiorient} that $D$ is not $(=p)$-invertible.
    Otherwise, we obtain from Lemma~\ref{lem:equiorient} that there exists an $(=p)$-decycling family for $D$.
\end{proof}

\section{Bounds on the \texorpdfstring{$(=p)$}{(=p)}-inversion number}
\label{sec:bounds}

Once we know whether an oriented graph $D$ is $(=p)$-invertible or not, two different questions arise.
The first one is the following.
If $D$ is not $(=p)$-invertible, how close to an acyclic oriented graph can it be made ?
Of course, for this question, we need a measure of closeness to acyclicity. 
A natural one is the feedback-arc-set number.
\begin{proposition}
    Let $p$ be an odd integer such that $p\geq 3$. 
    Every oriented graph of order $n \geq p+2$ can be transformed
    using $(=p)$-inversions in an oriented graph with feedback-arc-set number at most $\frac{1}{2}\lceil n/2\rceil$.
\end{proposition}
\begin{proof}
    Let $D$ be an oriented graph of order $n$.
    Let $T$ be a tournament of order $n$ such that $D$ is an oriented subgraph of $T$.
    Since an arc-reversal changes the parity of the out-degree of exactly two vertices, one could reverse a set $R$ of at most $\frac{1}{2}\lceil n/2\rceil$ arcs in $T$ to obtain a tournament $T'$ having exactly $\lceil n/2\rceil$ out-even vertices.
    Now, by Corollaries~\ref{cor:p-odd-3mod4} and~\ref{cor:p-odd-1mod4}, $T'$ is $(=p)$-invertible, so there is family $\cal X$ of $(=p)$-sets such that
    $\Inv(T'; {\cal X})$ is acyclic. Then $R$ is a feedback arc-set of $\Inv(T; {\cal X})$, and $R\cap A(D)$ a feedback arc-set of $\Inv(D; {\cal X})$.
\end{proof}

The above proposition is tight. Indeed, for some odd positive integer $k$, consider a $k$-diregular tournament $T$, that is $d_T^-(v)=d_T^+(v)=k$ holds for every vertex $v\in V(T)$.
Its number of vertices is $n=2k+1$.
Let $T'$ be a tournament obtained from $T$ using $(=p)$-inversions for some odd $p$.
Since $T$ has no out-even vertices, neither  does $T'$.
Now a transitive tournament on $n$ vertices has $\lceil n/2\rceil$ out-even vertices, and an arc-reversal changes the parity of the out-degree of two vertices.
Thus $\fas(T') \geq \frac{1}{2}\lceil n/2\rceil$.

\bigskip
The second question is the following.
If $D$ is $(=p)$-invertible, what is the $(=p)$-inversion number of $D$?
We begin in \Cref{ssec:bound-size} by bounding the $(=p)$-inversion number by the number of arcs. 
Then, in \Cref{ssec:bound-fas-even}, we give bounds for even $p$, in terms of the feedback arc set number and the number of vertices.


\subsection{First bound in terms of the size }\label{ssec:bound-size}

The algebraic point of view used in Subsection~\ref{subsec:carac}, and in particular Equation~\eqref{eq:pi_is_correct_linear_algebra}, has an immediate consequence on the number of $(=p)$-inversions to go from an oriented graph to another (when it is possible):


\begin{theorem}
    Let $G$ be a graph,
    and let $\vec{G}_1, \vec{G}_2$ be two orientations of $G$.
    For every family $X_1, \dots, X_\ell \subseteq V(G)$ such that $\vec{G}_1 = \Inv(\vec{G}_2; X_1, \dots, X_\ell)$,
    there exists $X'_1, \dots, X'_{\ell'} \in \{X_1, \dots, X_\ell\}$ such that $\vec{G}_1 = \Inv(\vec{G}_2; X'_1, \dots, X'_{\ell'})$ and
    \[
        \ell' \leq |E(G)|.
    \]
\end{theorem}

\begin{proof}
    Let $\mathbf{a} \in \mathbb{F}_2^{E(G)}$ be defined by 
    \[
        \mathbf{a}_e =
        \begin{cases}
            0 & \textrm{if $e$ has the same orientation in $\vec{G}_1$ and $\vec{G}_2$} \\
            1 & \textrm{otherwise.}
        \end{cases}
    \]
    for every $e \in E(G)$.
    Moreover, for every $X \subseteq V(G)$, let $\mathbf{i}_X \in \mathbb{F}_2^{E(G)}$ be defined
    by
    \[
        (\mathbf{i}_X)_{uv} =
        \begin{cases}
            1 & \textrm{if $u,v \in X$} \\
            0 & \textrm{otherwise.}
        \end{cases}
    \]
    for every $uv \in E(G)$.
    Then we have, for every $X_1, \dots, X_\ell \subseteq V(G)$,
    \[
        \mathbf{a} = \sum_{j \in [\ell]} \mathbf{i}_{X_j} \Leftrightarrow \vec{G}_1 = \Inv(\vec{G}_2; X_1, \dots, X_\ell).
    \]

    Let $X_1, \dots, X_\ell \subseteq V(G)$ 
    such that $\vec{G}_1 = \Inv(\vec{G}_2; X_1, \dots, X_\ell)$.
    Consider $X_1, \dots, X'_{\ell'} \in \{X_1, \dots, X_\ell\}$ such that $\vec{G}_1 = \Inv(\vec{G}_2; X'_1, \dots, X'_{\ell'})$ and $\ell'$ is minimal. 
    We claim that $\ell' \leq |E(G)|$.
    Suppose for contradiction that $\ell' > |E(G)|$.
    Since $\ell' > |E(G)| = \dim \mathbb{F}_2^{E(G)}$, the family $(\mathbf{i}_{X_j})_j$ is dependent over $\mathbb{F}_2^{E(G)}$, so
    there exists $J \subseteq [\ell']$ nonempty such that $\sum_{j \in J} \mathbf{i}_{X'_j} = \mathbf{0}$.
    It follows that 
    \[
        \mathbf{a} = \sum_{j \in [\ell'] \setminus J} \mathbf{i}_{X'_j},
    \]
    and so $\vec{G}_1 = \Inv(\vec{G}_2; (X'_j)_{j \in [\ell'] \setminus J})$, contradicting the minimality of $\ell'$.
\end{proof}

This theorem directly implies Theorem~\ref{thm:borneA} stating that
$\inv^{= p}(D) \leq |A(D)|$ for every $(=p)$-invertible oriented graph $D$.

\subsection{Improved bounds for even \texorpdfstring{$p$}{p}}\label{ssec:bound-fas-even}

\subsubsection{Bound in terms of the feedback-arc-set number and its consequences}

 The aim of this subsection is to prove the upper bound (i) of Theorem~\ref{thm:upper-fas}, which is more precisely restated in Theorem~\ref{thm:ub-fas}, and its consequences.
 
\begin{theorem}\label{thm:ub-fas}
    Let $p \geq 4$ be an even integer and let $D$ be an oriented graph with at least $p + 2$ vertices. Then, 
    \[
        \inv^{= p}(D) \leq (2p-2)(\fas(D) + 1)
    \]
\end{theorem}
\begin{proof}
    Let $D_1$ be the acyclic oriented graph obtained by reversing every arc of a minimum feedback arc set of $D$ and let $D_2$ be an acyclic oriented graph obtained from $D_1$ by reversing exactly one arc. 
    Then $D$ disagrees with one of $D_1$, $D_2$ in a set $E'$ of arcs such that $|E'|$ is even and $|E'| \leq \fas(D) + 1$. 
    
    By Propositions~\ref{claim:2adjacentedges} and~\ref{claim:2edges},
    one can reverse the arcs of $|E'|$ two by two using at most $(4p - 4)\frac{|E'|}{2}$ $(=p)$-inversions.  Thus $\inv^{= p}(D) \leq (2p-2)(\fas(D) + 1)$.   
\end{proof}

A consequence of Theorem~\ref{thm:ub-fas} is that, when $p$ is even,
$\inv^{=p}(n)$ and $\inv^{\leq p}(n)$ are close to each other.

\closeinv*

\begin{proof}
    Let $n$ be an integer with $n \geq p+2$.
    The lower bound $\inv^{=p}(n) \geq \inv^{\leq p}(n)$ is clear.
    For the upper bound, consider
    a tournament $T$ on $n$ vertices,
   and let $A$ be a set of $p$ vertices in $T$.
   Consider a decycling $(\leq p)$-family $\mathcal{X}$ of $T-A$ of size at most $\inv^{\leq p}(|V(T-A)|) \leq \inv^{\leq p}(n)$.
    For every $X \in \mathcal{X}$, consider the $(=p)$-set obtained from $X$ by adding $p-|X|$ vertices of $A$, 
    and let $\mathcal{X}'$ be the family of all these sets.
    Observe that $\Inv(T,\mathcal{X}') - A = \Inv(T-A,\mathcal{X})$, which is acyclic.
    Hence, $\fas(\Inv(T,\mathcal{X}')) \leq \binom{p}{2} + (n-p)p$.
    It follows from Theorem~\ref{thm:ub-fas} that $\inv^{=p}(\Inv(T;\mathcal{X}')) \leq (2p-2)\left(\binom{p}{2} + (n-p)p +1\right) = \bigO(n)$,
    and so $\inv^{=p}(T) \leq |\mathcal{X}'| + \inv^{=p}(\Inv(T;\mathcal{X}')) \leq \inv^{\leq p}(n) + \bigO(n)$. 
\end{proof}

Another consequence of Theorem~\ref{thm:ub-fas} is that, for a fixed even positive integer $p$ and an oriented graph $D$ on at least $p+2$ vertices,
$\inv^{=p}(D)$ and $\inv^{\leq p}(D)$ are linearly tied.

\evenbound*

\begin{proof}
The first inequality clearly holds. Moreover, if $\fas(D)=0$, then the second inequality also trivially satisfied. We may hence suppose that $\fas(D)\geq 1$.   Note that $\inv^{\leq p}(D) \geq \frac{\fas({D})}{\binom{p}{2}}$, and using this we argue as follows.
    \begin{align*}   
        \inv^{\leq p}(D) (4p -4)\binom{p}{2}  & \geq (4p -4)\binom{p}{2} 
         \frac{\fas({D})}{\binom{p}{2}} \\
        & \geq (4p-4)\fas(D) \\
        & \geq (4p - 4)\frac{\fas(D) + 1}{2} \\
        & \geq \inv^{= p}(D) &&\textrm{by Theorem~\ref{thm:ub-fas}}. \qedhere
    \end{align*}
\end{proof}

A third consequence of Theorem~\ref{thm:ub-fas} is the following bound which, for even $p$ and dense digraphs, is better than the one of Theorem~\ref{thm:borneA}. 

\borneA*

This theorem is an immediate consequence of  Theorem~\ref{thm:ub-fas} and the following lemma.

\begin{lemma}\label{lem:A/p}
    Let $p$ be an integer greater than $1$, and let $D$ be an oriented graph of order $n$. 
    In at most $|A(D)|/(p-1)$ $(=p)$-inversions, we can transform $D$ into an oriented graph $D'$ with $\fas(D') \leq (p-2)n -\frac{3}{4}p^2 +\frac{7p}{4}$.   
\end{lemma}
\begin{proof}
    Let $(v_1, \dots , v_n)$ be an ordering of the vertices of $D$ and set $D=D_1$. For $k=1, \dots, n-p$, we do the following.
    \begin{center}
        \begin{minipage}{0.85\textwidth}
            Let $B_k$ be the set of backwards arcs with head $v_k$ in $D_k$.
            Let $u_1, \dots , u_t$ be the tails of the arcs in $B_k$. (Note that $\{u_1, \dots , u_t\}\subseteq \{v_{k+1}, \dots , v_n\}$.) 
            Set $s_k=\lfloor t/(p-1)\rfloor$, and for $i=1, \dots, s_k$, let $X_i=\{v_k\}\cup \{u_j \mid  (i-1)(p-1) +1 \leq j\leq i(p-1)\}$.
            We then invert all the $X_i$ (which are all $p$-sets) to obtain $D_{k+1}$.
        \end{minipage}
    \end{center}
    
    Observe that for each $k\in [n-p]$, at step $k$, we reversed all backward arcs with head $v_k$ except possibly $t - s_k(p-1)\leq p-2$. Moreover $v_k$ is in no inversions in further steps.
    So in $D'$, there are at most $p-2$ backward arcs with head $v_k$.
    Thus in $D'$, there are at most $(p-2)(n-p)$ backward arcs with head in $\{v_{1}, \dots , v_{n-p}\}$. 
    
    Now consider a median order $(w_{n-p+1}, \dots , w_n)$ of $D_{n-p+1}\langle \{v_{n-p+1}, \dots , v_n\}\rangle$. It has at most $\frac{1}{2}\binom{p}{2}$ backward arcs.
    Thus, considering the number of backward arcs in $D'$ for the ordering $(v_1, \dots , v_{n-p}, w_{n-p+1}, \dots , w_n)$,
    we obtain:
    \begin{align*}
            \fas(D') & \leq (p-2)(n-p) + \frac{1} {2}\binom{p}{2} \\
             & = (p-2)n -\frac{3}{4}p^2 +\frac{7}{4}p.
    \end{align*}

    The number of inversions made to go from $D$ to $D'$ is
    $\sum_{k=1}^{n-p}s_k \leq 1+\sum_{k=1}^{n-p}|B_k|/(p-1)$.
    Now each arc in $B_k$ is incident to $v_k$ and a vertex of $\{v_{k+1}, \dots , v_n\}$, so $\sum_{k=1}^{n-p}|B_k| \leq |A(D)|$.
    So the number of inversions is at most $|A(D)|/(p-1)$.
\end{proof}

\subsubsection{Improved bounds for dense oriented graphs}

 The aim of this subsection is to prove the upper bounds (ii) and (iii) of Theorem~\ref{thm:upper-fas}. 

Let us first prove the upper bound (ii). We need some preliminary result. The first one an easy and well known proposition. Its proof follows directly from Exercise 10.1 of Lovasz' book~\cite{Lov93} which states that every graph with more than $3(n-1)/2$ edges has two vertices joined by three independent paths.

\begin{proposition}\label{prop:no-odd-cycles}
    A graph of order $n \geq 1$ with no even cycles has at most $\big\lfloor \frac{3(n-1)}{2}\big\rfloor$ edges.
\end{proposition}

\begin{proposition}\label{prop:evencycle-1}
    Let $p\geq 3$ and $\ell\geq 2$ be two integers, let $D$ be an orientation  of a graph $G$ of order at least $p+2$, let $x_0, \dots , x_{2\ell-1}$ be $2\ell$ distinct vertices of $G$ and let $C=(x_0, \dots , x_{2\ell-1} , x_0)$ be a cycle (not necessarily contained in $G$). 
    With $4\ell$ $(=p)$-inversions, we can reverse the orientation of the edges in $E(C)\cap E(G)$ but no other edges. 
\end{proposition}
\begin{proof}
    For each $i \in [\ell]$, let $C_i$ be the cycle $(x_i, x_{i+1}, x_{i+\ell+1}, x_{i+\ell}, x_i)$ with indices modulo $2\ell$.
    By Proposition~\ref{claim:4cycle}, one can reverse the orientation of the edges of each $E(C_i)\cap E(G)$ but no others using four $(=p)$-inversions. Doing so for each $i$, every edge of type $x_{i+\ell}x_i$ has its orientation reversed twice and every edge of $E(C)\cap E(G)$ has its orientation reversed exactly once.
    Thus, those $4\ell$ inversions reverse the orientation of the edges in $E(C)\cap E(G)$ but no other edges. 
\end{proof}

We are now ready to prove the following more precise version of the upper bound (ii) of Theorem~\ref{thm:upper-fas}.
\begin{theorem}\label{thm:bound-2fas}
    Let $p$ be an even integer greater than $2$, $k \geq 1$ an integer and $D$ an oriented graph of order $n\geq p+2$.
    Then $\inv^{= p}(D) \leq 2\fas(D) + 2pn$.
    
\end{theorem}

\begin{proof}
    Let $D_1$ be the acyclic oriented graph obtained by reversing every arc of a minimum feedback arc set of $D$ and let $D_2$ be an acyclic oriented graph obtained from $D_1$ by reversing exactly one arc.
    Then there exists an acyclic oriented graph $D^*\in \{D_1,D_2\}$ that disagrees with $D$ in a set $F'$ of arcs such that $|F'|$ is even and $|F'| \leq \fas(D) + 1$.
    
    Let $F'_1$ be a maximal subset of $F'$ which is forms the union of arc-disjoint even cycles.
    By Proposition~\ref{prop:no-odd-cycles}, $|F'\setminus F'_1| \leq \lfloor \frac{3(n-1)}{2}\rfloor$.
    Let $P_2$ be a maximum set of disjoint pairs of adjacent arcs in $F'\setminus F'_1$, let $F'_2$ be the set of arcs in pairs of $P_2$, and $F'_3=F'\setminus (F'_1\cup F'_2)$.
    By our choice of $P_2$, $F'_3$ is a matching so $|F'_3|\leq n/2$.  Moreover $|F'_3|$ is even because $F'$ was even and $F'_1$ and  $F'_2$ are also even by definition.
    We first reverse the arcs of $F'_1$ using $4\ell$ $(=p)$-inversions per $2\ell$-cycle by Proposition~\ref{prop:evencycle-1}, 
    which gives $2|F'_1|$ inversions. 
    We then reverse the arcs of $F'_2$ using $2p-2$ $(=p)$-inversions per pair of adjacent edges by Proposition~\ref{claim:2adjacentedges}, which can be averaged as $p-1$ $(=p)$-inversions per arc. 
    Finally, we reverse the arcs of $F'_3$ using $4p-4$ $(=p)$-inversions per pair of non-adjacent arcs by Proposition~\ref{claim:2edges}, so in average $2p-2$ $(=p)$-inversions per arc. 
    Hence $D$ can be transformed  into the acyclic oriented graph $D^*$ using 
    $2|F'_1| + (p-1)|F'_2| + (2p-2)|F'_3|$ 
    $(=p)$-inversions.
    Since $|F'_2\cup F'_3| \leq \frac{3(n-1)}{2}$, and $|F'_3|\leq n/2$, we get
    
    \begin{align*}   
\inv^{= p}(D) & \leq  2|F'_1| + (p-1)|F'_2| + (2p-2)|F'_3|\\
& \leq  2|F'| + (p-3)|F'_2| + (2p-4)|F'_3|\\
& \leq  2\fas(D) + 2 + (p-3)|F'_2 \cup F'_3| + (p-1)|F'_3|\\
& \leq 2\fas(D) + 2 + (p-3)\frac{3(n-1)}{2} + (p-1)\frac{n}{2}\\
& \leq 2\fas(D) + (2p-5)n - \frac{3}{2}p  + \frac{13}{2}.
\end{align*}

\end{proof}

Let us now prove the upper bounds (iii) of Theorem~\ref{thm:upper-fas}. We need the following theorem of Lam and Verstaete~\cite{LaVe05} which
generalizes results of Kovari~\cite{kovari1954} and Reiman~\cite{reiman1958} on graphs without $C_4$ and improves on the Even-Circuit Theorem due to Bondy and Simonovits \cite{BoSi1974}.

\begin{theorem}[Lam and Verstaete~\cite{LaVe05}]\label{thm:noC2k}
    Let $G$ be a graph of order $n$ and $k\geq 2$ an integer.
    If $G$ has no even cycle of lengths in $\{4, \dots ,2k\}$, then $|E(G)| \leq  \frac{1}{2}n^{1+1/k} + 2^{k^2}n$.
\end{theorem}
This theorem is asymptotically tight when $k\in \{2,3,5\}$ as the polarity graphs of order $n$ of generalized $(k+1)$-gons \cite{LUW99} have
$\frac{1}{2}n^{1+1/k} + \Omega(n)$ edges and no even cycles of length at most $2k$. 

This theorem is useful in combination with the following generalization of
Propositon~\ref{claim:4cycle}.

\begin{proposition}\label{prop:evencycle}
    Let $p\geq 3$ and $\ell\geq 2$ be two integers, and let $D$ be an oriented graph of order at least $p+2\ell-2$.   With $2\ell$ $(=p)$-inversions, we can reverse the arcs of a $2\ell$-cycle and no other arc. 
\end{proposition}
\begin{proof}
    Consider a $2\ell$ -cycle $C=(x_1, \dots , x_{2\ell}, x_1)$.
    Let $X$ be a set of $p-2$ vertices of $V(D)\setminus V(C)$. Such a set exists because the order of $D$ is at least $p+2\ell-2$.
    Let us invert the sets $X\cup \{x_1, x_2\}$,  $X\cup \{x_2, x_3\}$, \dots, and $X\cup \{x_{2\ell}, x_1\}$.
    Doing so, all arcs are reversed an even number of times except those of $C$.
\end{proof}

Similarly to Theorem~\ref{thm:bound-2fas} with Proposition~\ref{prop:evencycle} instead of Proposition~\ref{prop:evencycle-1}, and Theorem~\ref{thm:noC2k}
instead of Proposition~\ref{prop:no-odd-cycles}, one proves the following more precise statement of the upper bound (ii) of Theorem~\ref{thm:upper-fas}.

\begin{theorem}\label{thm:bound-fas}
    Let $p$ be an even integer greater than $2$, $k \geq 1$ be an integer, and $D$ be an oriented graph of order $n\geq p+2k-2$.
    Then $\inv^{= p}(D) \leq \fas(D) + \frac{p-2}{2}n^{1+1/k} + \left( (p-2) 2^{k^2}+ \frac{p-1}{2}\right) n +1$.
\end{theorem}

\subsubsection{Improved bounds for very dense oriented graphs}\label{subsec:very-dense}

In this subsection, we prove the upper bound (iv) of Theorem~\ref{thm:upper-fas}. 

We need some preliminaries.
For two sets $X_1,X_2$, we denote by $X_1 \Delta X_2$ the symmetric difference of $X_1$ and $X_2$,
that is the set $(X_1 \cup X_2) \setminus (X_1 \cap X_2)$. Further, for some set $X$, we use $\binom{X}{2}$ for the collection of two-element subsets of $X$. Yuster proved in~\cite{yuster2023tournamentinversion} the following theorem.
Since it was not stated explicitly in~\cite{yuster2023tournamentinversion}, we provide a short proof.

\begin{theorem}\label{thm:reversing_a_dense_set_of_edges}
    Let $p$ be an integer with $p \geq 2$.
    There exist constants $\alpha_p$ and $\epsilon_p$ with $\epsilon_p>0$ such that the following holds.
    Let $n$ be a positive integer,  and let $G$ be a graph with  $V(G)=[n]$. 
    There exists an integer $\ell$ with $\ell \leq \frac{|E(G)|}{\lceil p/2\rceil \cdot \lfloor p/2\rfloor}$,
    and a family $X_1, \dots, X_\ell \subseteq [n]$ of $(=p)$-sets such that
    \[
        \left|E(G) \Delta \binom{X_1}{2} \Delta \dots \Delta \binom{X_\ell}{2}\right| \leq \alpha_p \cdot n^{2-\epsilon_p}.
    \]
\end{theorem}

The proof relies on the K{\H{o}}v{\'a}ri-S{\'o}s-Tur{\'a}n Theorem.
\begin{theorem}[K{\H{o}}v{\'a}ri, S{\'o}s, and Tur{\'a}n~\cite{kovariCM3}]\label{thm:kst}
    Let $s,t$ be  positive integers with $t \leq s$,
    and let $G$ be a graph on $n$ vertices.
    If $|E(G)| > 4s^{\frac{1}{t}} n^{2-\frac{1}{t}}$, then $K_{s,t}$ is a subgraph of $G$.
\end{theorem}

\begin{proof}[Proof of \Cref{thm:reversing_a_dense_set_of_edges}]
    Let $s = 2\lfloor p/2\rfloor$, $t = 2\lceil p/2\rceil$,
    $\epsilon_p = \frac{1}{t}$, and $\alpha_p = 4s^{\frac{1}{t}}$.
    We prove the result by induction on $|E(G)|$, the result holding trivially if $|E(G)| \leq \alpha_p n^{2-\frac{1}{t}}$.

   Now assume $|E(G)| > \alpha_p n^{2-\frac{1}{t}}$.
    By Theorem~\ref{thm:kst}, there is a subgraph $H$ of $G$ isomorphic to $K_{s,t}$.
    Let $(B,C)$ be the bipartition of $H$ with $|B|=s$ and $|C|=t$.
    Let $(B_1,B_2)$ (resp. $(A_1,A_2)$) be an arbitrary partition of $B$ (resp. $C$) into two sets of size $\lfloor p/2\rfloor$ (resp. $\lceil p/2\rceil$),
    and let $(X_1,X_2,X_3,X_4) = (B_1 \cup C_1, B_1 \cup C_2, B_2 \cup c_1, B_2 \cup C_2)$.
    Observe that $X_1, X_2, X_3, X_4$ all have size $p$,
    and $E(G) \Delta \binom{X_1}{2} \Delta \binom{X_2}{2} \Delta \binom{X_3}{2} \Delta \binom{X_4}{2} = E(G) \setminus E(H)$.
    Then, by the induction hypothesis applied to $E(G) \setminus E(H)$,
    there exists $X_5, \dots, X_\ell \subseteq [n]$ with $\ell-4 \leq \frac{|E(G)| - st}{\lfloor p/2\rfloor \lceil p/2\rceil}$
    such that
    \[
        \left|E(G)\Delta \binom{X_1}{2} \Delta \dots \Delta \binom{X_\ell}{2}\right| =
        \left|(E(G) \setminus E(H)) \Delta \binom{X_5}{2} \Delta \dots \Delta \binom{X_\ell}{2}\right| \leq \alpha_p \cdot n^{2-\epsilon_p}.
    \]
    Since $st = 4 \lfloor p/2\rfloor \lceil p/2\rceil$, we deduce that $\ell \leq  \frac{|E(G)|}{\lfloor p/2\rfloor \lceil p/2\rceil}$,
    and so $X_1, \dots, X_\ell$ is as claimed.
\end{proof}

\begin{theorem}\label{thm:upper-opt}
    Let $p$ be a fixed even integer with $p \geq 2$. Then,
    $\inv^{=p}(D) \leq \frac{\fas(D)}{\lceil p/2\rceil \cdot \lfloor p/2\rfloor} + o(n^2)$ for every oriented graph $D$ of order $n$.
\end{theorem}
\begin{proof}
    Let $D$ be an oriented graph of order $n$ and $F$ a minimum feedback arc set of $D$. 
    By Theorem~\ref{thm:reversing_a_dense_set_of_edges} applied to $F$, 
    there is an $(=p)$-family ${\cal X}$ of size at most $\ell \leq  \frac{|F|}{\lceil p/2\rceil \cdot \lfloor p/2\rfloor}$ of subsets of $V(D)$ such that
    the arcs reversed in the inversion of $\cal X$ differs from $F$ in $o(n^2)$ arcs. 
    Hence, setting $D_1=\Inv(D, {\cal X})$, we have $\fas(D_1) = o(n^2)$. 
    Thus, by Theorem~\ref{thm:ub-fas}, $\inv^{=p}(D_1) = o(n^2)$. 
    Hence $\inv^{=p}(D) \leq |{\cal X}| + o(n^2) \leq \frac{\fas(D)}{\lceil p/2\rceil \cdot \lfloor p/2\rfloor} + o(n^2)$.
\end{proof}

Observe that \Cref{thm:upper-opt} directly implies $(iv)$ of \Cref{thm:upper-fas}, and using the fact that $\fas(D) \leq \frac{|A(D)|}{2}$, we also get the following bound in terms of $|A(D)|$.
\begin{corollary}\label{cor:opt-edge}
    Let $p$ be a fixed even integer with $p \geq 2$.
    Then $\inv^{= p}(D) \leq \frac{|A(D)|}{2\lceil p/2\rceil \cdot \lfloor p/2\rfloor} + o(n^2)$ for every oriented graph $D$ of order $n$.
\end{corollary}

\subsection{Improved bounds for odd \texorpdfstring{$p$}{p}}\label{ssec:bound-fas-odd}

A natural question is whether the results established in the previous subsection for even values of $p$ admit analogues when $p$ is odd.

We begin by considering an analogue to Theorem~\ref{thm:ub-fas}, which would be the existence of a function $f_p$ such that $\inv^{= p}(D)\leq f_p(\fas(D))$. It turns out that such a function does not exist even when $D$ is a tournament and $p = 3$.
\begin{theorem}\label{thm:no-f}
    Let $p$ be an odd integer with $p \geq 3$.
    There is no function $f_p$ such that $\inv^{= p}(T)\leq f_p(\fas(T))$ holds for every $(=p)$-invertible tournament $T$. 
\end{theorem}
\begin{proof}
    For every positive integer $n$, let $TT_n$ be the transitive tournament on $n$ vertices $(v_1, \ldots, v_n)$ ordered acyclically and let $T_n$ be the tournament obtained from $TT_n$ by reversing the arc $v_1v_n$ (so $\inv^{\leq p}(T_n) = \fas(D) = 1$).

    Note that for any even $n$, $T_n$ has exactly $n/2$ out-even vertices, which by Theorem~\ref{cor:p-odd-3mod4} or~\ref{cor:p-odd-1mod4} yields that $T_n$ is indeed $(=p)$-invertible. 
    Towards a contradiction, 
    suppose there is a function $f_p$ such that for every even $n$ with $n \geq p+2$, we have $\inv^{= p}(T_n)\leq f_p(\fas(T_n))=f_p(1)$. 
    Then, consider the first such $n$ satisfying $\inv^{= p}(T_n) \leq \frac{n-3}{p}$, and let ${\cal X}$ be a decycling $(=p)$-family of $T_n$ of cardinality at most $\frac{n - 3}{p}$. 
    Then, we may take a vertex $v_i$ with $i \in [2,n-1]$ that does not belong to any set in ${\cal X}$. 
    Thus ${\cal X}$ is also a decycling $(=p)$-family of $T_{n-1} = T_n - v_i$, meaning $T_{n-1}$ is also $(=p)$-invertible.
    But, since $n - 1$ is odd, $T_{n-1}$ has less than $\lceil (n-1)/2 \rceil$ out-even vertices, which is in contradiction with \Cref{cor:p-odd-3mod4} or~\ref{cor:p-odd-1mod4}.
\end{proof}

Since $\inv^{\leq p}(D) \leq \fas (D) \leq \binom{p}{2} \inv^{\leq p}(D)$, \Cref{thm:no-f} directly implies that Corollary~\ref{cor:even-bound} admits no analogue for every odd $p \geq 3$.
\begin{corollary}
    \label{cor:no-g}
    Let $p$ be an odd integer with $p \geq 3$.
    There is no function $g_p$ such that $\inv^{= p}(T)\leq g_p(\inv^{\leq p}(T))$ for all $(=p)$-invertible tournament $T$.
\end{corollary}

Regarding bounds in terms of the number of arcs, we can ask whether Theorem~\ref{thm:borneAbis} admits a counterpart (for odd $p$).
 \begin{problem}
     Let $p \geq 3$ be an odd integer. 
     Does there exist a constant $C_p$ such that $\inv^{= p}(D) \leq \frac{|A(D)|}{p-1} + C_p\cdot n$ for every invertible oriented graph $D$ of order $n \geq p + 2$ ?
 \end{problem}

It would also be interesting to obtain bounds in terms of $\fas$ similar to Theorems~\ref{thm:bound-2fas} and~\ref{thm:bound-fas}.
\begin{problem}
     Let $p \geq 3$ be an odd integer, and $k\geq 1$ an integer.
    \begin{itemize}
        \item Does there exist a constant $C_p$ such that $\inv^{= p}(D) \leq 2\fas(D) + C_p n$ for every invertible oriented graph of order $n \geq p + 2$ ?
        \item Does there exist a constant $\gamma_p$ such that $\inv^{= p}(D) \leq \fas(D) + \gamma_p\cdot n^{1+1/k}$ for every invertible oriented graph of order $n \geq p + 2k-2$ ?
    \end{itemize}
 \end{problem}

We can then also ask whether our improved bounds for dense graphs,  Theorem~\ref{thm:upper-opt} and Corollary~\ref{cor:opt-edge}, admit extensions.
\begin{problem}
     Let $p\geq 3$ be an odd integer.
    \begin{itemize}
        \item Is is true that $\inv^{=p}(D) \leq \frac{\fas(D)}{\lceil p/2\rceil \cdot \lfloor p/2\rfloor} + o(n^2)$ for every invertible oriented graph $D$ of order $n$ ?
        \item Is it true that $\inv^{=p}(n) = \frac{\binom{n}{2}}{2 \lceil p/2\rceil \cdot \lfloor p/2\rfloor} + o(n^2)$ for every invertible oriented graph $D$ of order $n$?
    \end{itemize}

\end{problem}

\section{Complexity of computing the \texorpdfstring{$(= p)$}{p}-inversion and \texorpdfstring{$(\leq p)$}{p}-inversion numbers}\label{sec:complexity}

\subsection{NP-completeness of {\sc Tournament \texorpdfstring{$(= p)$}{(=p)}-Inversion} and {\sc Tournament \texorpdfstring{$(\leq p)$}{(<=p)}-Inversion}}\label{subsec:NP-hardness}

Given a tournament $T$ and an integer $k$, recall that deciding whether $T$ has a feedback arc set of size at most $k$ is NP-complete~\cite{alonSJDM20,charbitCPC16}.
This means that {\sc Tournament $(= p)$-Inversion} and  {\sc Tournament $(\leq p)$-Inversion} are NP-complete when $p=2$.
We show this generalizes to any $p$.

\begin{theorem}\label{rzgu}
{\sc Tournament $(= p)$-Inversion} and {\sc Tournament $(\leq p)$-Inversion} are NP-complete, for any fixed integer $p$ greater than $2$.
\end{theorem}

Our reduction will be from the hypergraph edge-colourability problem.
A {\bf hypergraph} $\mathcal{H}$ consists of a vertex set $V(\mathcal{H})$ and a multiset $\mathcal{E}(\mathcal{H})$ of subsets of $V(\mathcal{H})$, called the {\bf hyperedges} of $\mathcal{H}$. For some nonnegative integer $k$, we say that $\mathcal{H}$ is {\bf $k$-uniform} if $|e|=k$ holds for every $e \in \mathcal{E}(\mathcal{H})$ and {\bf $k$-regular} if every $v \in V(\mathcal{H})$ is contained in exactly $k$ hyperedges.
For some positive integer $k$, a {\bf $k$-edge-colouring} of a hypergraph $\mathcal{H}$ is a mapping $\phi\colon E(\mathcal{H})\rightarrow [k]$. If $\phi(e)\neq \phi(e')$ holds for each pair of distinct hyperedges $e,e'$ with $e \cap e' \neq \emptyset$, then we say that $\phi$ is a {\bf proper} $k$-edge-colouring. We say that a hypergraph is {\bf $k$-edge-colourable} if it admits a proper $k$-edge-colouring. 

A hypergraph $\mathcal{H}$ is {\bf $p$-special} if it satisfies the following four properties:
\begin{enumerate}[label=(\roman*)]
\item[(S1)] $\mathcal{H}$ is $p$-uniform;
\item[(S2)] $\mathcal{H}$ is $3$-regular;
\item[(S3)] $|e \cap e'|\leq 1$ for all distinct hyperedges $e,e' \in E(\mathcal{H})$;
\item[(S4)] for all triples $(v,v',v'')$ of distinct vertices, if each of $\{v,v'\},\{v,v''\}$, and $\{v',v''\}$ is contained in a hyperedge of $\mathcal{H}$, then $\{v,v',v''\}$ is contained in a hyperedge of $\mathcal{H}$.
\end{enumerate}

We denote by {\sc $p$-Special Hypergraph 3-Edge-Colourability}, in short $(p,3)$-SHEC, the problem of deciding whether a given $p$-special hypergraph $\mathcal{H}$ is 3-edge-colourable.

The hardness of $(p,3)$-SHEC will be crucial for our reduction proving Theorem~\ref{rzgu}. Note that the problem  $(2,3)$-SHEC is exactly the problem of deciding whether a cubic, triangle-free graph is 3-edge-colourable. This problem is well known to be NP-hard. (See for example Theorem 4 in \cite{CAI199115}.)

\begin{proposition}\label{setdrzftugzhuij}
    $(2,3)$-SHEC is NP-hard.
\end{proposition}
In the following, we give a slightly technical proof for the fact that this hardness result generalizes for arbitrary $p \geq 3$.
\begin{lemma}
    $(p,3)$-SHEC is NP-hard for every $p \geq 2$.
\end{lemma}
\begin{proof}
    By Proposition~\ref{setdrzftugzhuij}, the statement holds for $p=2$. Now consider an integer $p \geq 3$. Inductively, we may suppose that $(p-1,3)$-SHEC is NP-hard. We will show the hardness of $(p,3)$-SHEC by a reduction from $(p-1,3)$-SHEC.

    Let $\mathcal{H}$ be a $(p-1)$-special hypergraph. Let us create a particular hypergraph $\mathcal{H}'$. For every $v \in V(\mathcal{H})$ and every $(i,j)\in [p]^2$, we let $V(\mathcal{H}')$ contain a vertex $v_{i,j}$. 
    Next, for every $e \in E(\mathcal{H})$ and every $(i,j)\in [p]^2$, we let $V(\mathcal{H}')$ contain a vertex $x_{i,j}^{e}$. For every $e \in E(\mathcal{H})$ and every $(i,j)\in [p]^2$, we let $E(\mathcal{H}')$ contain a hyperedge $e_{i,j}$ which consists of $\{v_{i,j} \mid v \in e\}$ and $x_{i,j}^{e}$. 
    Further, for every $e \in E(\mathcal{H})$ and every $i \in [p]$, we let $E(\mathcal{H}')$ contain a hyperedge $f_{e,i}$ consisting of $\{x_{i,j}^{e} \mid j \in [p]\}$. Finally, for every $e \in E(\mathcal{H})$ and every $j \in [p]$, we let $E(\mathcal{H}')$ contain a hyperedge $f'_{e,j}$ consisting of $\{x_{i,j}^{e}\mid i \in [p]\}$. This finishes the description of $\mathcal{H}'$. 
    Observe that $|V(\mathcal{H}')|=p^2(|V(\mathcal{H})|+|E(\mathcal{H})|)$ and, as $\mathcal{H}'$ is 3-regular, $|E(\mathcal{H}')|\leq 3|V(\mathcal{H}')|$. 
    Hence the size of $\mathcal{H}'$ is polynomial in the size of $\mathcal{H}$.
\begin{claim}
    $\mathcal{H}'$ is $p$-special.
\end{claim}
\begin{subproof} 
    We need to prove that $\mathcal{H}'$ satisfies (S1), (S2), (S3) and (S4).

    \medskip
    
    It follows immediately from the fact that $\mathcal{H}$ is $(p-1)$-uniform  and by construction that $\mathcal{H}'$ satisfies (S1).

\medskip
    
    Similarly, as $\mathcal{H}$ is 3-regular, is easy to see that $\mathcal{H}'$ is 3-regular, so $\mathcal{H}'$ satisfies (S2). 
 
\medskip
       
    We next prove (S3). Let $g$, $g'$ be distinct hyperedges in $E(\mathcal{H}')$. First suppose that $g=f_{e,i}$ for some $e \in E(\mathcal{H})$ and some $i \in [p]$. If $g'=f_{e',j}$ for some $e' \in E(\mathcal{H})$ and $j \in [p]$, we obtain that $(e,i)\neq (e',j)$ by assumption. It follows by construction that $g\cap g'=\emptyset$. Next suppose that $g'=f'_{e',j}$ for some $e' \in E(\mathcal{H})$ and some $j \in [p]$. We obtain by construction that $g \cap g'=\{x_{i,j}^{e}\}$ if $e'=e$ and $g \cap g'=\emptyset$, otherwise. Now suppose that $g'=e'_{i',j}$ for some $e' \in E(\mathcal{H})$ and $(i',j)\in [p]^2$. It follows by construction that $g \cap g'=\{x^{e}_{i,j}\}$ if $(e',i')=(e,i)$, and $g \cap g'=\emptyset$, otherwise. The case of $g=f'_{e,j}$ for some $e \in E(\mathcal{H})$ and some $j \in [p]$ can be handled similarly. We may hence suppose that $g=e_{i,j}$ for some $e \in E(\mathcal{H})$ and $(i,j)\in [p]^2$. Using similar arguments, we may suppose that $g'=e'_{i',j'}$ for some $e' \in E(\mathcal{H})$ and $(i',j')\in [p]^2$. If $(i,j)\neq (i',j')$, we obtain that $g \cap g'=\emptyset$ by construction. If $(i,j)= (i',j')$, we obtain that $g \cap g'=\{v_{i,j} \mid v \in e \cap e'\}$. As $\mathcal{H}$ is $(p-1)$-special, we have $|e \cap e'|\leq 1$ and so $|g \cap g'|\leq 1$. Hence $\mathcal{H}'$ satisfies (S3).
 
\medskip
       
    Let us now prove (S4). Consider three distinct vertices $v,v',v''$ of $V(\mathcal{H}')$ such that each of $\{v,v'\}$, $\{v,v''\}$, and $\{v',v''\}$ is contained in a hyperedge of $E(\mathcal{H}')$. 
    
    First suppose that $v=x_{i,j}^{e}$ for some $(i,j)\in [p]^2$ and $e \in E(\mathcal{H})$. As $\{v,v'\}$ is contained in a hyperedge of $E(\mathcal{H})$, we obtain that either $v'=x_{i',j}^{e}$ for some $i' \in [p]-i$, $v'=x_{i,j'}^{e}$ for some $j' \in [p]-j$ or $v'=u'_{i,j}$ for some $u' \in e$. Similarly, we have that either $v''=x_{i'',j}^{e}$ for some $i'' \in [p]-i$, $v''=x_{i,j''}^{e}$ for some $j'' \in [p]-j$ or $v''=u''_{i,j}$ for some $u'' \in e$.
    If $v'=x_{i',j}^{e}$ for some $i' \in [p]-i$ and $v''=x_{i'',j}^{e}$ for some $i'' \in [p]-\{i,i'\}$, we obtain that $\{v,v',v''\}\subseteq f'_{e,j}$. If $v'=x_{i',j}^{e}$ for some $i' \in [p]-i$ and $v''=x_{i,j''}^{e}$ for some $j'' \in [p]-j$, we obtain by construction that no hyperedge in $E(\mathcal{H}')$ contains $\{v',v''\}$, a contradiction. If $v'=x_{i',j}^{e}$ for some $i' \in [p]-i$ and $v''=u''_{i,j}$ for some $u'' \in e$, we obtain by construction that no hyperedge in $E(\mathcal{H}')$ contains $\{v',v''\}$, a contradiction. If $v'=u'_{i,j}$ for some $u' \in e$ and $v''=u''_{i,j}$ for some $u'' \in e-u'$, we obtain that $\{v,v',v''\}\subseteq e_{i,j}$. The remaining cases for the different forms of $v'$ and $v''$ when $v=x_{i,j}^{e}$ are symmetric.

    We may hence suppose that $v=u_{i,j}$ for some $u \in V(\mathcal{H})$ and $(i,j)\in [p]^2$. Similarly, we may suppose that $v'=u'_{i',j'}$ for some $u' \in V(\mathcal{H})$ and $(i',j')\in [p]^2$ and $v''=u''_{i'',j''}$ for some $u'' \in V(\mathcal{H})$ and $(i'',j'')\in [p]^2$. If $(i,j)\neq (i',j')$, we obtain by construction that $E(\mathcal{H})$ does not contain a hyperedge containing $\{v,v'\}$, a contradiction. We hence obtain $(i',j')=(i,j)$. We similarly have $(i'',j'')=(i,j)$. 

    By construction, we obtain that each of $\{u,u'\},\{u,u''\}$, and $\{u',u''\}$ is contained in a hyperedge of $E(\mathcal{H})$. As $\mathcal{H}$ is $(p-1)$-special, we obtain that there exists some $e \in E(\mathcal{H})$ such that $\{u,u',u''\}\subseteq e$. By construction, it follows that $\{v,v',v''\}\subseteq e_{i,j}$. Hence $\mathcal{H}'$ satisfies (S4). 
    \end{subproof}

    It remains to prove that $\mathcal{H}'$ is 3-edge-colourable if and only if $\mathcal{H}$ is. First suppose that $\mathcal{H}'$ is 3-edge-colourable. Then it admits a proper $3$-edge-colouring $\phi'$. Let $\phi$ be the $3$-edge-colouring of $\mathcal{H}$ defined by $\phi(e)=\phi'(e_{1,1})$ for every $e \in E(\mathcal{H})$. Let $e,e' \in E(\mathcal{H})$ with $e \cap e' \neq \emptyset$. Then, by construction, we have $e_{1,1}\cap e'_{1,1}\neq \emptyset$. As $\phi'$ is a proper $3$-edge-colouring, $\phi'(e_{1,1})\neq \phi'(e'_{1,1})$ and hence by construction $\phi(e)\neq \phi(e')$. It follows that $\phi$ is a proper $3$-edge-colouring of $\mathcal{H}$.

    Now suppose that $\mathcal{H}$ is 3-edge-colourable. Then it admits a proper $3$-edge-colouring $\phi$. We now define a $3$-edge-colouring $\phi'$ of $\mathcal{H}'$ as follows. First, for every $e \in E(\mathcal{H})$ and $(i,j)\in [p]^2$, we set $\phi'(e_{i,j})=\phi(e)$. Now consider some $e \in E(\mathcal{H})$ and let $\alpha,\beta$ be the unique integers with $\alpha<\beta$ and $\{\alpha,\beta\}=[3]\setminus\{\phi(e)\}$.
    We then set $\phi'(f_{e,i})=\alpha$ for all $i \in [p]$ and $\phi'(f'_{e,j})=\beta$ for all $j \in [p]$. This finishes the description of $\phi'$. 
    Let us prove that $\phi'$ is proper.
    Consider some $g,g' \in E(\mathcal{H}')$ with $g \cap g' \neq \emptyset$. If $g=f_{e,i}$ for some $e \in E(\mathcal{H})$ and some $i \in [p]$ or $g=f'_{e,j}$ for some $e \in E(\mathcal{H})$ and some $j \in [p]$, we obtain that $\phi'(g)\neq \phi'(g')$ by construction. We may hence suppose that $g=e_{i,j}$ for some $e \in E(\mathcal{H})$ and some $(i,j)\in [p]^2$. Similarly, we may suppose that $g'=e'_{i',j'}$ for some $e' \in E(\mathcal{H})$ and some $(i',j')\in [p]^2$. If $(i,j)\neq (i',j')$, we obtain that $g\cap g'=\emptyset$ by construction, a contradiction. We hence have $(i',j')=(i,j)$ and $e\neq e'$. By construction and as $g \cap g'\neq \emptyset$, we obtain that $e \cap e' \neq \emptyset$. As $\phi$ is a proper $3$-edge-colouring of $\mathcal{H}$, we obtain that $\phi(e)\neq \phi(e')$ and hence by construction that $\phi'(g)\neq \phi'(g')$. Hence $\phi'$ is a proper $3$-edge-colouring of $\mathcal{H}'$. This finishes the proof.
\end{proof}

For the main proof of Theorem~\ref{rzgu}, we further need the following simple observation.

\begin{proposition}\label{xstrdztfzgh}
    Let $D$ be an oriented graph and $\mathcal{X}$ be a decycling $(\leq p)$-family of size $k$ for some positive integers $k$ and $p$. 
    If there exists an arc $uv \in A(D)$ and $pk$ directed cycles $C_1,\ldots,C_{pk+1}$ in $D$ such that $V(C_i)\cap V(C_j)=\{u,v\}$ for all distinct $i,j \in [pk+1]$ and $uv \in A(C_i)$ for all $i \in [pk+1]$, then there exists $X \in \mathcal{X}$ such that $\{u,v\}\subseteq X$.
\end{proposition}
\begin{proof}
    Suppose otherwise. Then $uv \in A(\Inv(D;\mathcal{X}))$. Consider $Y=\bigcup_{X \in \mathcal{X}}X$. Observe that $|Y|\leq kp$. Hence, there exists some $i \in [pk+1]$ such that $(V(C_i)\setminus \{u,v\})\cap Y=\emptyset$. 
    Hence $C_i$ is an oriented subgraph of $\Inv(D;\mathcal{X})$, a contradiction to $\mathcal{X}$ being decycling.
\end{proof}

We shall also use the following notation : given two disjoint sets of vertices $A$ and $B$ in an oriented graph $D$, we write $A\Ra B$ if $A(D)$ contains all possibles arcs with tail in $A$ and head in $B$.
We are now ready to proceed to the main proof of Theorem~\ref{rzgu}.

\begin{proof}[Proof of Theorem~\ref{rzgu}]
    Let $\mathcal{H}$ be a $(p-1)$-special hypergraph and let $(v_1,\ldots,v_n)$ be an arbitrary ordering of $V(\mathcal{H})$. 
    We further set $k=|E(\mathcal{H})|$. We now create a tournament $T$. 
    We first let $V(T)$ contain $V(\mathcal{H})$. Further, for all distinct $i,j \in [n]$ with $i<j$, we let $A(T)$ contain $v_jv_i$ if $\{v_i,v_j\}\subseteq e$ for some $e \in E(\mathcal{H})$, and $v_iv_j$, otherwise. Next, for $i \in [n]$, we let $V(T)$ contain a set $W_i$ of $pk+1$ vertices and we let $A(T)$ contain arcs such that $T\langle W_i\rangle$ is a transitive tournament. We set $W=\bigcup_{i \in [n]}W_i$. Next, for $i \in [n]$, we add arcs such that $\{v_1,\ldots,v_{i-1}\} \Ra W_i$ and $W_i \Ra \{v_{i},\ldots,v_n\}$ in $T$. Further, for all distinct $i,j \in [n]$, with $i<j$ we let $W_i\Ra W_j$ in $T$.   Finally, we let $V(T)$ contain a set $Z=\{z_1,z_2,z_3\}$. We let $Z\Ra W$ and $V(\mathcal{H})\Ra Z$ in $T$. Finally, we add some arcs such that $T\langle Z\rangle$ is a transitive tournament. This finishes the description of $T$. Observe that the size of $T$ is polynomial in the size of $\mathcal{H}$.

We shall prove that the following three statements are equivalent:
    \begin{itemize}
        \item[(i)] $\mathcal{H}$ is 3-edge-colourable.
        \item[(ii)] $\inv^{= p}(T)\leq k$.
        \item[(iii)] $\inv^{\leq p}(T)\leq k$.
    \end{itemize}
 This equivalence directly implies both hardness results.

(ii) trivially implies (iii). We shall prove that (i) implies (ii), and (iii) implies (i).

\medskip
(i) $\Rightarrow$ (ii).

Suppose that $\mathcal{H}$ admits a proper $3$-edge-colouring $\phi$. For every $e \in E(H)$, we set $X_e=e\cup\{z_{\phi(e)}\}$. Observe that $|X_e|=p$, as $\mathcal{H}$ is $(p-1)$-uniform. Set $\mathcal{X}=(X_e \mid e \in E(\mathcal{H}))$ and $X=\bigcup_{e \in E(\mathcal{H})}X_e$. Observe that $\mathcal{X}$ contains exactly $k$ sets. Now let $T'=\Inv(T;\mathcal{X})$. We now prove that $T'$ is a transitive tournament, which yields $\inv^{= p}(T)\leq k$.

    First consider  some $i \in [n]$ and some $j \in [3]$. Then, by construction, we have that $A(T)$ contains the arc $v_iz_j$. As $\mathcal{H}$ is 3-regular and $\phi$ is a proper 3-edge-colouring of $\mathcal{H}$, there exists exactly one $e \in E(\mathcal{H})$ with $v_i \in e$ and $\phi(e)=j$. It follows that $\{v_i,z_j\}\subseteq X_e$ and $\{v_i,z_j\} \nsubseteq X_{e'}$  for all $e' \in E(\mathcal{H})\setminus \{e\}$. It follows that $A(T')$ contains the arc $z_jv_i$. 
    Hence $Z\Ra V(\mathcal{H})$ in $T'$.
    Further, as $X \cap W=\emptyset$, we have $Z \Ra W$ in $T'$. Next, as $|X_e\cap Z|\leq 1$ for all $e \in E(\mathcal{H})$, we obtain that $T'\langle Z\rangle=T\langle Z\rangle$, and so $T'\langle Z\rangle$ is a transitive tournament. 
     It hence suffices to prove that $T'-Z$ is a transitive tournament.

    Next consider some distinct $i,j \in [n]$ with $i<j$. If there exists no hyperedge in $E(\mathcal{H})$ containing $\{v_i,v_j\}$, then, by construction, we have that $A(T)$ contains $v_iv_j$ and $\{v_i,v_j\}\setminus X_e \neq \emptyset$ for all $e \in E(\mathcal{H})$. It follows that $A(T')$ contains the arc $v_iv_j$. 
    If there exists a hyperedge $e \in E(\mathcal{H})$ containing $\{v_i,v_j\}$, then, by the property (S3) of $(p-1)$-special hypergraphs, $e$ is the only hyperedge containing $\{v_i,v_j\}$. It follows by construction that $A(T)$ contains $v_jv_i$, $\{v_i,v_j\}\subseteq X_e$, and $\{v_i,v_j\} \nsubseteq X_{e'}$ for all ${e'} \in E(\mathcal{H})\setminus \{e\}$. Thus $A(T')$ contains $v_iv_j$.  Therefore $T'\langle V\mathcal({H})\rangle$ is transitive with acyclic ordering $(v_1, \dots , v_n)$.
    Now since we have  $X \cap W=\emptyset$, all the arcs incident to a vertex of $W$ are unchanged when inverting the sets of $\mathcal{X}$. Hence, in $T'$, $\{v_1,\ldots,v_{i-1}\} \Ra W_i$ and $W_i \Ra \{v_{i},\ldots,v_n\}$ for all $i\in [n]$, $W_i\Ra W_j$ for all distinct $i,j \in [n]$.
    Hence $T'-Z$ is a transitive tournament, which implies that $T'$ is a transitive tournament.
    
    \medskip
  (iii) $\Rightarrow$ (i). 
  
  Now suppose that $T$ has a decycling $(\leq p)$-family $\mathcal{X}$ of size at most $k$. 
    For $i \in [k]$, let $F_i$ be the set of arcs of $A(T)$ for which there exist exactly $i$ sets in $\mathcal{X}$ each of which contain both endvertices of the arc and let $F=\bigcup_{i \in [k]}F_i$. Further, let $A_1$ be the set of all arcs oriented from $V({\mathcal H})$ to $Z$ in $A(T)$, and $A_2$ those arcs in $A(T)$ of the form $v_jv_i$ for some $i,j \in [n]$ with $i<j$.
    \begin{claim}\label{sedrftgzhuij}
        $F=F_1=A_1\cup A_2$ and $|X|=p$ for all $X \in {\mathcal{X}}$.
    \end{claim}
    \begin{subproof}
        First consider an arc $v_iz_j \in A_1$ for some $i \in [n]$ and $j \in [3]$. Observe that for every $w \in W_1$, we have that $z_jwv_iz_j$ is a triangle in $T$. It thus follows by Proposition~\ref{xstrdztfzgh} and $|W_1|=pk+1$ that $v_iz_j \in F$. Hence $A_1 \subseteq F$.
    
        Next consider some distinct $i,j \in [n]$ with $i<j$ such that $v_jv_i \in A_2$. Observe that for every $w \in W_j$, we have that $v_iwv_jv_i$ is a directed 3-cycle in $T$. Thus $v_jv_i \in F$ by Proposition~\ref{xstrdztfzgh}. Hence $A_2 \subseteq F$.
        
        As $A_1\cap A_2 =\emptyset$, we get $|F|\geq |A_1|+|A_2|$. Moreover as $\mathcal{H}$ is $(p-1)$-special, a double counting of the vertex-hyperedge incidences yields $3n=k(p-1)$.  Thus
        
        \begin{align*}
            k\binom{p}{2}&\geq \sum_{X \in \mathcal{X}}{\binom{|X|}{2}}\\
            &= \sum_{i=1}^ki|F_i|\\
            &\geq |F|\\
            &\geq |A_1|+|A_2|\\
            &= 3n+k\binom{p-1}{2} \\
            &=k(p-1) +k\binom{p-1}{2}\\
            &=k\binom{p}{2}.
        \end{align*}
        Hence equality holds throughout and the statement follows.
    \end{subproof}
    Now consider some $X \in \mathcal{X}$. By Claim~\ref{sedrftgzhuij}, we have $|X|=p$ and $X \subseteq Z \cup V(\mathcal{H})$. Further, Claim~\ref{sedrftgzhuij} yields $F\cap A(T\langle Z\rangle)=\emptyset$ and hence $|X \cap Z|\leq 1$. Now consider some distinct $i,j \in [n]$ such that $\{v_i,v_j\} \subseteq X$. By Claim~\ref{sedrftgzhuij}, one of $v_iv_j$ and $v_jv_i$ is contained in $A_2$, and hence, by construction and (S3), there exists a unique hyperedge $e \in E(\mathcal{H})$ such that $\{v_i,v_j\} \subseteq e$. As $v_i$ and $v_j$ were chosen to be arbitrary elements of $X \cap V(\mathcal{H})$ and by the property (S4), we have $X \cap V(\mathcal{H}) \subseteq e$.
    We obtain $p=|X|=|X \cap Z|+|X \cap V(\mathcal{H})|\leq 1+|e|=1+(p-1)=p$. Hence equality holds throughout which yields $X=e \cup z_j$ for some $j \in [3]$.  As $F=F_1$ by Claim~\ref{sedrftgzhuij}, we obtain that for every $e \in E(\mathcal{H})$, there exists at most one $X \in \mathcal{X}$ with $e \subseteq X$. As $|\mathcal{X}|=k=|E(\mathcal{H})|$, we obtain that for every $e \in E(\mathcal{H})$, there exists exactly one $j \in [3]$ such that $e \cup z_j \in \mathcal{X}$.
    
    We are now ready to define a $3$-edge-colouring $\phi$ of $\mathcal{H}$. For every $e \in E(\mathcal{H})$, we let $\phi(e)$ be the unique integer $j \in [3]$ such that $e \cup z_j \in \mathcal{X}$. 
    Let us prove that $\phi$ is proper. Consider some $v \in V(\mathcal{H})$ and suppose for the sake of a contradiction that there exist two distinct hyperedges $e,e' \in E(\mathcal{H})$ containing $v$ with $\phi(e)=\phi(e')$. It follows by construction that $\mathcal{X}$ contains both the sets $e \cup z_{\phi(e)}$ and $e' \cup z_{\phi(e)}$. We obtain that $vz_{\phi(e)}\in F\setminus F_1$, a contradiction to Claim~\ref{sedrftgzhuij}.
\end{proof}

\subsection{\texorpdfstring{$W[1]$}{W[1]}-hardness with parameter \texorpdfstring{$p$}{p}}\label{subsec:W1}

Let us recall the \mcc\ problem, which is a $W[1]$-hard problem, that often allows easy reduction in many hardness proofs.

\defparproblem{\mcc\ (MCC)}{A graph $G$, along with a partition $(V_1,\ldots,V_k)$ of $V(G)$}{$k$}{Does $G$ contain a clique with exactly one vertex from each $V_i$?}

The hardness of \mcc\ derives by a straightforward reduction from $k$-\textsc{Clique}, shown by Fellows et al.~\cite{fellowsParameterizedComplexityMultipleinterval2009}.
\begin{proposition}[Fellows et al.~\cite{fellowsParameterizedComplexityMultipleinterval2009}]
    {\sc MultiColoured Clique} is W[1]-hard.
\end{proposition}

We are now ready to show the hardness of {\sc Bounded Size Inversion} and {\sc Prescibed Size Inversion} parameterized by $p$ by reduction from \mcc.
As noted in the introduction, it is implied by the following
theorem.

\Wone*

\begin{proof}
    We prove this statement by a reduction from MCC. 
    Let $(G,(V_1,\ldots,V_k))$ be an instance of MCC. 
    We now create an instance $(D,p)$ of the inversion problem. For $i \in [k]$, let $q_i=|V_i|$ and let $v_i^1,\ldots,v_i^{q_i}$ be an arbitrary enumeration of $V_i$. For every $i \in [k]$ and $j \in [q_i]$, we let $V(D)$ contain two vertices $w_i^j$ and $x_i^j$ and for every $i \in [k]$, we let $A(D)$ contain arcs such that $w_i^1x_i^1w_i^2\ldots w_i^{q_i}x_i^{q_i}w_i^1$ is a directed cycle, which we denote by $C_i$.
    Further, we denote by $F$ the set of all ordered pairs $(v_{i_1}^{j_1},v_{i_2}^{j_2})$ with $1 \leq i_1<i_2 \leq k, j_1\in [q_{i_1}], j_2 \in [q_{i_2}] $ and such that $E(G)$ does not contain an edge linking $v_{i_1}^{j_1}$ and $v_{i_2}^{j_2}$. We let $V(D)$ contain a set $Z$ that contains a vertex $z_f$ for every $f \in F$. Further, for every $f=(v_{i_1}^{j_1},v_{i_2}^{j_2})\in F$, we let $A(D)$ contain the arcs $w_{i_1}^{j_1}z_f,z_fw_{i_2}^{j_2}$, and $w_{i_1}^{j_1}w_{i_2}^{j_2}$. Finally, we set $p=2k$. This finishes the description of $(D,p)$. For an illustration, see Figure~\ref{firstsecond}. 

    \begin{figure}[hbtp]
        \centering
        \includegraphics{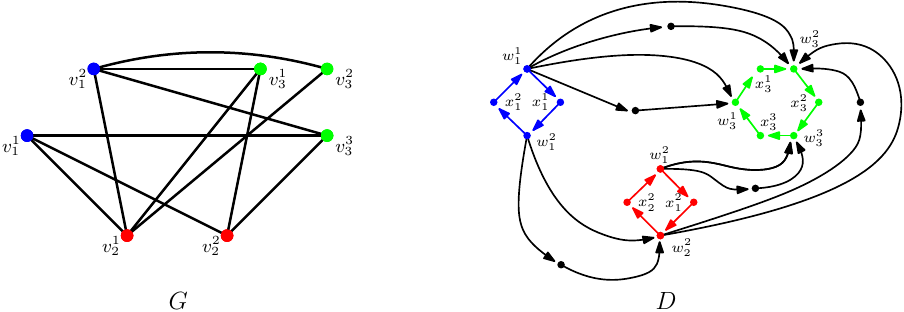} 
        \caption{An illustration of an instance $(G,(V_1,V_2,V_3))$ of MCC and the corresponding oriented graph $D$. The vertices of $V_1,V_2$, and $V_3$ are marked in blue, red, and green in $G$, respectively. For $i \in [3]$, the directed cycle $C_i$ in $D$ is marked in the same colour as $V_i$ in $G$. All vertices of $Z$ and arcs of $D$ incident to $Z$ are marked in black. The names of the vertices in $Z$ have been omitted due to space restrictions.}\label{firstsecond}
    \end{figure}
    It is easy to see that $(D,p)$ can be constructed in polynomial time from $(G,(V_1,\ldots,V_k))$ and clearly, $p$ is bounded by a computable function of $k$.

    We shall prove that the following three statements are equivalent:
    \begin{itemize}
        \item[(i)] $(G,(V_1,\ldots,V_k))$ is a yes-instance of MCC.
        \item[(ii)] $\inv^{= p}(D)\leq 1$.
              \item[(iii)] $\inv^{\leq p}(D)\leq 1$.
    \end{itemize}
This equivalence directly implies both hardness results.
    
(ii) trivially implies (iii). We shall prove that (i) implies (ii), and (iii) implies (i).

  \medskip
  (i) $\Rightarrow$ (ii).
    
    First suppose that $(G,(V_1,\ldots,V_k))$ is a yes-instance of MCC. Then for $i \in [k]$, there exists some $r_i \in [q_i]$ such that $G\langle \{v_1^{r_1},\ldots,v_k^{r_k}\}\rangle$ is a clique. Set $X=\{w_1^{r_1},x_1^{r_1},\ldots,w_k^{r_k},x_k^{r_k}\}$ and $D'=\Inv(D,X)$. Clearly, we have $|X|=2k=p$. It hence suffices to prove that $D'$ is acyclic.
    Note that $V(D) = Z \cup \bigcup_{i \in [k]} V(C_i)$, 
    and the strong components of $D$ are either a directed cycle $(C_i)_i$ or a single vertex.
    Since $\{v_1^{r_1},\ldots,v_k^{r_k}\}$ induces a clique in $G$, $X$ 
    induces exactly the arcs $(w_i^{r_i},x_i^{r_i})_{i \in [k]}$ in $D$. 
    For any $i$, the inversion of $w_i^{r_i}x_i^{r_i}$ yields that $D'\langle V(C_i) \rangle$ is acyclic. Moreover, observe that 
    no arc between two distinct strong components of $D$ is reversed in $D'$.
    This implies that $D'$ is acyclic, therefore $\inv^{= p}(D)\leq 1$.
        
 \medskip
  (iii) $\Rightarrow$ (i).    
  
  Suppose that $\inv^{\leq p}(D)\leq 1$. Then there exists a $(\leq p)$-set $X \subseteq V(D)$ such that $D'$ is acyclic where $D'=\Inv(D,X)$. As $C_i$ is a directed cycle for $i \in [k]$ and $V(C_{i_1})\cap V(C_{i_2})=\emptyset$ for all distinct $i_1,i_2 \in [k]$, we obtain that $2k=p\geq |X|=\sum_{i=1}^k|X \cap V(C_i)|+|X \cap Z|\geq \sum_{i=1}^k|X \cap V(C_i)|\geq \sum_{i=1}^k2=2k$. We obtain that equality holds throughout, so we have $X \cap Z=\emptyset$ and $|X \cap V(C_i)|=2$ for $i \in [k]$. Now consider some $i \in [k]$. As $C_i$ is a directed cycle and $D'$ is acyclic, we obtain that $X$ contains two consecutive vertices of $C_i$. As $|V(C_i)\cap X|=2$, this yields that there exists one unique $r_i \in [q_i]$ such that $w_i^{r_i}\in X$.

    We claim that $G\langle \{v_1^{r_1},\ldots,v_k^{r_k}\}\rangle$ is a clique. Suppose otherwise, so there exists some distinct $i_1,i_2 \in [k]$ with $i_1<i_2$ such that $E(G)$ does not contain an edge linking $v_{i_1}^{r_{i_1}}$ and $v_{i_2}^{r_{i_2}}$. It follows that $f=(v_{i_1}^{r_{i_1}},v_{i_2}^{r_{i_2}})\in F$. As $\{w_{i_1}^{r_{i_1}},w_{i_2}^{r_{i_2}}\}\subseteq X$ and $X \cap Z=\emptyset$, we obtain that $w_{i_1}^{r_{i_1}}z_pw_{i_2}^{r_{i_2}}w_{i_1}^{r_{i_1}}$ is directed cycle in $D'$, a contradiction. Hence $G\langle\{v_1^{r_1},\ldots,v_k^{r_k}\}\rangle$ is a clique and so $(G,(V_1,\ldots,V_k))$ is a yes-instance of MCC.
\end{proof}


\subsection{A \texorpdfstring{$O(k^2 p^3)$}{O(k2 p3)} kernel for {\sc Tournament Bounded Size Inversion} and {\sc Tournament Prescribed Size Inversion}} \label{subsec:kernel}

As mentioned in the introduction, the problem of deciding whether a tournament admits a feedback arc set of size $k$ has a $(2+\epsilon)k$ kernel for every $\epsilon > 0$ \cite{bessyKernel}. In our language, this kernel is a kernel for {\sc Tournament $(=2)$-Inversion} and {\sc Tournament $(\leq 2)$-Inversion}. 
This kernelization relies on reduction rules to output an equivalent instance of size subquadratic in $k$, then employs the polynomial-time approximation scheme of \cite{kenyon2007rank}. Recall that, given any fixed $\epsilon > 0$, a PTAS is an algorithm running in polynomial time and outputting a $(1+\epsilon)$-approximation for a minimization problem.
We now prove the existence of kernels for {\sc Tournament Bounded Size Inversion} and {\sc Tournament Prescribed Size Inversion}, also using the latter approximation to find a feedback arc set $F$ of size bounded by a function of $p$ and $k$.
Then, $T\setminus F$ is acyclic, and we consider the positions of the arcs of $F$ in $T$ to prune vertices ``in between'' while preserving an equivalent instance, as long as the ``gaps'' are large enough.

The main technical contribution is contained in the following lemma, which we use to ensure the soundness of our reduction rule.
\begin{lemma}\label{delvertex}
    Let $\epsilon>0$ be fixed, let $p\geq 2$ and $k$ be nonnegative integers and let $T$ be a tournament on at least $(2(1+\epsilon)k{\binom{p}{2}}+1)(pk+2)$ vertices.
    Then, in polynomial time, we can compute a tournament $T_1$ 
    with $|V(T_1)|<|V(T)|$ such that:
    \begin{itemize}
        \item $\inv^{\leq p}(T_1)\leq k$ if and only if $\inv^{\leq p}(T)\leq k$
        \item $\inv^{= p}(T_1)\leq k$ if and only if $\inv^{= p}(T)\leq k$
    \end{itemize}
\end{lemma}
\begin{proof}
    We first run the PTAS of \cite{kenyon2007rank} to yield a feedback arc set $F$ such that $|F| \leq (1+ \epsilon) \fas(T)$.
    
    First suppose that $|F| > (1+ \epsilon) k \binom{p}{2}$.
    We assert that $(T,p,k)$ is a no-instance. Indeed, suppose that $T$ admits a decycling family $(X_1,\ldots,X_{k})$ with $|X_i|\leq p$ for $i \in [k]$. Then, we obtain in particular that the set $F'$ of arcs of $A(T)$, for which there exists an odd number of indices in $[k]$ containing both its endvertices, forms a feedback arc set for $T$. We thus have that $(1+\epsilon)\fas(T)\leq (1+\epsilon)|F'|\leq (1+\epsilon)k{\binom{p}{2}}<|F|$, a contradiction.
    We then output a small no-instance $(T_1,p,k)$, for example  $T_1$ is a tournament of size $3 (kp+1) < |T|$ obtained by adding arcs arbitrarily to the disjoint union of $p k+1$ directed triangles. 
    
    Suppose now that $|F| \leq (1+ \epsilon) k \binom{p}{2}$.
    Then, we can assume that $F$ is arc-minimal, because an arc of $F$ whose deletion from $F$ still yields a feedback arc set can be found and removed in quadratic time.
    We then compute an acyclic ordering $\sigma$ of $T \setminus F$, that is, a mapping $\sigma:V(T)\rightarrow |V(T)|$ such that $\sigma(u)<\sigma(v)$ for all $uv \in A(T)\setminus F$. Clearly, such a mapping can be obtained in polynomial time.
    Since $F$ is arc-minimal, note that any $ab \in F$ is such that $\sigma(b) < \sigma(a)$. 
    Let $S \subseteq V(T)$ be the set of endvertices of arcs in $F$. Observe that $|S| \leq 2|F|\leq 2(1 + \epsilon) k \binom{p}{2}$.
    A set $I \subseteq V(T)$ is called an {\it interval} of $(V(T),\sigma)$ if $I=\sigma^{-1}([i,j])$ for some $i < j \leq |V(T)|$. As $|V(T)|\geq (2(1+\epsilon)k{\binom{p}{2}}+1)(pk+2)$, there exists an interval $I \subseteq V(T)$ such that $I\cap S=\emptyset$ and $|I|\geq pk+2$. Let $z\in I$. We now let $T_1=T-z$. Clearly, we have that $(T_1,k,p)$ can be computed in polynomial time and $|V(T_1)|<|V(T)|$. In order to show that $(T_1,p,k)$ has the desired properties, it hence suffices to prove the following statement.
    \begin{claim}
       Let $k' \leq k$ be a nonnegative integer and let $q_1,\ldots,q_{k'} \in [p]$. Then $T$ admits a decycling family $(X_1,\ldots,X_{k'})$ with $|X_i|=q_i$ for $i \in [k]$ if and only if $T_1$ admits a decycling family $(X_1,\ldots,X_{k'})$ with $|X_i|=q_i$ for $i \in [k']$.
    \end{claim}
    \begin{proofclaim}
        First suppose that $T$ admits a decycling family $(X_1,\ldots,X_{k'})$ with $|X_i|=q_i$ for $i \in [k']$. As $|I|\geq kp+2>kp\geq \sum_{i\in [k']}q_i=\sum_{i\in [k']}|X_i|\geq |\bigcup_{i \in [k']}X_i|$, we obtain that there exists some $z'\in I \setminus \bigcup_{i \in [k']}X_i$. Observe that $T-z'$ is isomorphic to $T_1$ and $(X_1,\ldots,X_{k'})$ is a decycling family for $T-z'$. Hence the statement follows. 

        Now suppose that $T_1$ admits a decycling family $(X_1,\ldots,X_{k'})$ with $|X_i|=q_i$ for $i \in [k']$. As $|I\setminus \{z\}|\geq kp+1>kp\geq \sum_{i\in [k']}q_i=\sum_{i\in [k']}|X_i|\geq |\bigcup_{i \in [k']}X_i|$, we obtain that there exists some $y\in I \setminus (\bigcup_{i \in [k']}X_i\cup \{z\})$. Let $T_2$ be the tournament obtained from $T_1$ by adding a vertex $z'$ with $N^-_{T_2}(z')=N^-_{T_1}(y)$ and $N^+_{T_2}(z')=N^+_{T_1}(y)\cup \{y\}$. Observe that $T_2$ is isomorphic to $T$. It hence suffices to prove that $(X_1,\ldots,X_{k'})$ is a decycling family for $T_2$. Let $T_1'=\inv(T_1,X_1,\ldots,X_{k'})$ and $T_2'=\inv(T_2,X_1,\ldots,X_{k'})$. Suppose for the sake of a contradiction that $T_2'$ is not acyclic. Since $T_2'$ is a tournament, we obtain that $T_2'$ contains a directed triangle $Q$. As $(X_1,\ldots,X_{k'})$ is a decycling family for $T_1'$ and $T_1'=T_2'-z'$, we obtain that $z'\in V(Q)$. First suppose that $y\in V(Q)$ and let $x$ be the unique vertex in $V(Q)\setminus \{z',y\}$. Then, as $\{z',y\}\cap (\bigcup_{i \in [k']}X_i)=\emptyset$ and by construction, we have $\{z',y\}\subseteq N_{T_2'}^-(x)$ or $\{z',y\}\subseteq N_{T_2'}^+(x)$. This contradicts $Q$ being a directed triangle. We hence have that $y\in V(T)\setminus V(Q)$. Let $Q'$ be the directed subgraph of $T_2'$ induced by $V(Q)\setminus \{z'\}\cup \{y\}$. Observe that $N^+_{T_2'}(y)=N^+_{T_2}(y)=N^+_{T_2}(z')=N^+_{T_2'}(z')$, so $Q'$ is a directed triangle. As $T_2'- \{z'\}=T_1'$ by construction, we obtain that $Q'$ is also a directed subgraph of $T_1'$. This contradicts $T_1'$ being acyclic. 
    \end{proofclaim}
\end{proof}

We are now ready to conclude the main theorem of this section, showing both kernels at once.

\kernel*
\begin{proof}
Without loss of generality, we may suppose that $\epsilon \leq \frac{1}{2}$.
   Let $(T,p,k)$ be an instance of {\sc Tournament Bounded Size Inversion} or {\sc Tournament Prescribed Size Inversion}. First, if $p \leq 1$ or $k=0$, we can clearly decide whether $(T,p,k)$ is a yes-instance and output an appropriate kernel in polynomial time. Next, if $kp< \frac{30}{\epsilon}$, we solve the instance by a brute force approach in $n^{O(pk)}$ (polynomial) time, and output an appropriate kernel. We may hence suppose that $kp\geq  \frac{30}{\epsilon}$.
   Next, observe that if $|V(T)|\leq (2(1+\frac{\epsilon}{2})k{\binom{p}{2}}+1)(pk+2)$, since $kp\geq  \frac{30}{\epsilon}$ and $\epsilon \leq \frac{1}{2}$, we obtain that: 
   \begin{align*}
      (2(1+\frac{\epsilon}{2})k{\binom{p}{2}}+1)(pk+2)&\leq (1+\frac{\epsilon}{2})k^2p^3+5kp^2+kp+2\\
      &\leq (1+\frac{\epsilon}{2})k^2p^3+15kp^2\\
      &\leq (1+\frac{\epsilon}{2})k^2p^3+\frac{\epsilon}{2}k^2p^3\\
      &=(1+\epsilon)k^2p^3.
   \end{align*}
    It follows that $(T,p,k)$ is a trivial kernel of appropriate size.
    
    We may thus assume in the following that $|V(T)|> (2(1+\frac{\epsilon}{2})k{\binom{p}{2}}+1)(pk+2)$.
    Then, starting with $T_1=T$, we recursively apply \Cref{delvertex} and update $T_1$ with the resulting tournament as long as $|V(T_1)|> (2(1+\frac{\epsilon}{2})k{\binom{p}{2}}+1)(pk+2)$. This eventually yields an instance $(T_1,p,k)$ {\sc Tournament Bounded Size Inversion} or {\sc Tournament Prescribed Size Inversion}, respectively.
    Then, $|V(T_1)| \leq (2(1+\frac{\epsilon}{2})k{\binom{p}{2}}+1)(pk+2)$ which by the inequalities above yields $|V(T_1)| \leq (1+ \epsilon)k^2 p^3$ as required.
    Since each application of \Cref{delvertex} runs in polynomial time, and the size of $T$ decreases by at least 1, the entire algorithm for our theorem runs in polynomial time.
\end{proof}

\Cref{thm:kernel} implies that {\sc Tournament Bounded Size Inversion} and {\sc Tournament Precribed Size Inversion} are FPT with parameters $p$ and $k$.

\section*{Acknowledgements}
This work was partially supported by the french Agence Nationale de la Recherche under contract Digraphs ANR-19-CE48-0013-01 and the Independent Research Fund Denmark under grant
DFF-7014-00037B. Caroline Silva was supported by FAPESP Proc. 2023/16755-2.

\bibliographystyle{abbrv}
\bibliography{biblio}

@PREAMBLE{ {\providecommand{\noopsort}[1]{}} }

@article{KlSv98,
author = {Klostermeyer, William F. and S\v{o}lt\'es, L\v{u}bom\'{\i}r},
title = {Hamiltonicity and reversing arcs in digraphs},
journal = {Journal of Graph Theory},
volume = {28},
number = {1},
pages = {13-30},
doi = {https://doi.org/10.1002/(SICI)1097-0118(199805)28:1<13::AID-JGT2>3.0.CO;2-I},
url = {https://onlinelibrary.wiley.com/doi/abs/10.1002/%28SICI%291097-0118%28199805%2928%3A1%3C13%3A%3AAID-JGT2%3E3.0.CO%3B2-I},
eprint = {https://onlinelibrary.wiley.com/doi/pdf/10.1002/%28SICI%291097-0118%28199805%2928%3A1%3C13%3A%3AAID-JGT2%3E3.0.CO%3B2-I},
year = {1998}
}

@article{erdosSetsConsistentArcs1965,
  title = {On {{Sets}} of {{Consistent Arcs}} in a {{Tournament}}},
  author = {Erd{\"o}s, P. and Moon, J. W.},
  year = 1965,
  month = apr,
  journal = {Canadian Mathematical Bulletin},
  volume = {8},
  number = {3},
  pages = {269--271},
  issn = {0008-4395, 1496-4287},
  doi = {10.4153/CMB-1965-017-1},
  urldate = {2025-06-01},
  copyright = {https://www.cambridge.org/core/terms},
  langid = {english}
}

@article{KlMc04,
title = {Homomorphisms and oriented colorings of equivalence classes of oriented graphs},
journal = {Discrete Mathematics},
volume = {274},
number = {1},
pages = {161-172},
year = {2004},
issn = {0012-365X},
doi = {https://doi.org/10.1016/S0012-365X(03)00086-4},
url = {https://www.sciencedirect.com/science/article/pii/S0012365X03000864},
author = {William F. Klostermeyer and Gary MacGillivray}
}

@article{HuWo01,
title = {Nearly-acyclically pushable tournaments},
author = {Jing Huang and Kathryn L.B. {Wood}},
journal = {Austral. J. Combin.},
volume = {23},
pages = {142-152},
year = {2001},
}

@article{HMW01,
title = {Pushing the cycles out of multipartite tournaments},
journal = {Discrete Mathematics},
volume = {231},
number = {1},
pages = {279-287},
year = {2001},
issn = {0012-365X},
doi = {https://doi.org/10.1016/S0012-365X(00)00324-1},
url = {https://www.sciencedirect.com/science/article/pii/S0012365X00003241},
author = {Jing Huang and Gary MacGillivray and Kathryn {L.B. Wood}}
}

@article{RiZ06,
author = {Rizzi, Romeo},
year = {2006},
month = {06},
pages = {1177-1188},
title = {Acyclically pushable bipartite permutation digraphs: An algorithm},
volume = {306},
journal = {Discrete Mathematics},
doi = {10.1016/j.disc.2005.11.027}
}

@INPROCEEDINGS{HeHu09,
 author={Heard, Scott and Huang, Jing},
 booktitle={2009 Fifth International Joint Conference on INC, IMS and IDC}, 
 title={Kernel-Perfection through the Push Operation},   year={2009},
 volume={},
  number={},
  pages={1955-1957},
  doi={10.1109/NCM.2009.386}
}

@article{HMY02,
title = {Pushing vertices in digraphs without long induced cycles},
journal = {Discrete Applied Mathematics},
volume = {121},
number = {1},
pages = {181-192},
year = {2002},
issn = {0166-218X},
doi = {https://doi.org/10.1016/S0166-218X(01)00299-2},
url = {https://www.sciencedirect.com/science/article/pii/S0166218X01002992},
author = {Jing Huang and Gary MacGillivray and Anders Yeo}
}

@article{FiRy95,
author = {Fisher, David C. and Ryan, Jennifer},
title = {Tournament games and positive tournaments},
journal = {Journal of Graph Theory},
volume = {19},
number = {2},
pages = {217-236},
doi = {https://doi.org/10.1002/jgt.3190190208},
url = {https://onlinelibrary.wiley.com/doi/abs/10.1002/jgt.3190190208},
eprint = {https://onlinelibrary.wiley.com/doi/pdf/10.1002/jgt.3190190208},
year = {1995}
}

@article{McWo00,
title = {Re-orienting tournaments by pushingvertices, },
journal = {Ars Combinatoria},
volume = {57},
pages = {33-47},
year = {2000},
author = {Gary MacGillivray and Kathryn L.B. {Wood}},
}

@article{Klo98,
author = {Klostermeyer, William},
year = {1998},
volume = {132},
pages = {205-218},
title = {An Analogue of Camion's Theorem in Squares of Cycles},
journal = {Congressus Numerantium}
}

@article{Klo99,
title = {Pushing vertices and orienting Edges},
journal = {Ars Combinatoria},
volume = {51},
pages = {65-75},
year = {1999},
author = {William F. Klostermeyer},
}

@article{fas-FPT,
author = {Chen, Jianer and Liu, Yang and Lu, Songjian and O'sullivan, Barry and Razgon, Igor},
title = {A fixed-parameter algorithm for the directed feedback vertex set problem},
year = {2008},
issue_date = {October 2008},
publisher = {Association for Computing Machinery},
address = {New York, NY, USA},
volume = {55},
number = {5},
issn = {0004-5411},
url = {https://doi.org/10.1145/1411509.1411511},
doi = {10.1145/1411509.1411511},
journal = {J. ACM},
month = nov,
articleno = {21},
numpages = {19}
}

@book{Lov93,
author = {L\'aszl\'o Lov\'asz},
title = {Combinatorial Problems and Exercises (Second Edition)},
publisher = {North-Holland},
edition = {Second Edition},
address = {Amsterdam},
year = {1993},
isbn = {978-0-444-81504-0},
doi = {https://doi.org/10.1016/B978-0-444-81504-0.50002-3},
url = {https://www.sciencedirect.com/science/article/pii/B9780444815040500023}
}

@article{LUW99,
title = {Polarities and $2k$-cycle-free graphs},
journal = {Discrete Mathematics},
volume = {197-198},
pages = {503-513},
year = {1999},
note = {16th British Combinatorial Conference},
issn = {0012-365X},
doi = {https://doi.org/10.1016/S0012-365X(99)90107-3},
url = {https://www.sciencedirect.com/science/article/pii/S0012365X99901073},
author = {Felix Lazebnik and Vasiliy A. Ustimenko and Andrew J. Woldar},
}

@article{LaVe05,
  author       = {Thomas Lam and
                  Jacques Verstra{\"{e}}te},
  title        = {A Note on Graphs Without Short Even Cycles},
  journal      = {The Electronic Journal of Combinatorics},
  volume       = {12},
  year         = {2005},
  url          = {https://doi.org/10.37236/1972},
  doi          = {10.37236/1972},
  note={\href{https://arxiv.org/abs/math/0503623}{\nolinkurl{arXiv:0503623}}}
}

@article{BoSi1974,
title = {Cycles of even length in graphs},
journal = {Journal of Combinatorial Theory, Series B},
volume = {16},
number = {2},
pages = {97--105},
year = {1974},
issn = {0095-8956},
doi = {https://doi.org/10.1016/0095-8956(74)90052-5},
url = {https://www.sciencedirect.com/science/article/pii/0095895674900525},
author = {J. A. Bondy and M. Simonovits}
}

@article{fellowsParameterizedComplexityMultipleinterval2009,
  title = {On the Parameterized Complexity of Multiple-Interval Graph Problems},
  author = {Fellows, Michael R. and Hermelin, Danny and Rosamond, Frances and Vialette, St{\'e}phane},
  year = {2009},
  journal = {Theoretical Computer Science},
  volume = {410},
  number = {1},
  pages = {53--61},
  issn = {0304-3975},
  doi = {10.1016/j.tcs.2008.09.065},
  urldate = {2025-02-18}
}

@inbook{karp1972,
  title = {{Reducibility among Combinatorial Problems}},
  ISBN = {9781468420012},
  url = {http://dx.doi.org/10.1007/978-1-4684-2001-2_9},
  DOI = {10.1007/978-1-4684-2001-2_9},
  booktitle = {Complexity of Computer Computations},
  publisher = {Springer US},
  author = {Karp,  Richard M.},
  year = {1972},
  pages = {85–103}
}

@article{alonSJDM20,
    author = {Alon, Noga},
    title = {{Ranking Tournaments}},
    journal = {SIAM Journal on Discrete Mathematics},
    volume = {20},
    number = {1},
    pages = {137--142},
    year = {2006},
    doi = {10.1137/050623905},
}

@article{charbitCPC16,
  author       = {Pierre Charbit and
                  St{\'{e}}phan Thomass{\'{e}} and
                  Anders Yeo},
  title        = {{The Minimum Feedback Arc Set Problem is NP-Hard for Tournaments}},
  journal      = {Combinatorics, Probability, and Computing},
  volume       = {16},
  number       = {1},
  pages        = {1--4},
  year         = {2007},
  url          = {https://doi.org/10.1017/S0963548306007887},
  doi          = {10.1017/S0963548306007887},
}

@article{delavega1983,
title = {On the maximum cardinality of a consistent set of arcs in a random tournament},
journal = {Journal of Combinatorial Theory, Series B},
volume = {35},
number = {3},
pages = {328--332},
year = {1983},
issn = {0095-8956},
doi = {https://doi.org/10.1016/0095-8956(83)90060-6},
url = {https://www.sciencedirect.com/science/article/pii/0095895683900606},
author = {Wenceslas {Fernandez de la Vega}}
}

@article{spencer1980,
  title = {Optimally ranking unrankable tournaments},
  volume = {11},
  ISSN = {1588-2829},
  url = {http://dx.doi.org/10.1007/BF02017965},
  DOI = {10.1007/bf02017965},
  number = {2},
  journal = {Periodica Mathematica Hungarica},
  publisher = {Springer Science and Business Media LLC},
  author = {Spencer,  Joel},
  year = {1980},
  pages = {131--144}
}

@article{spencer1971,
author = {Spencer, Joel},
title = {Optimal ranking of tournaments},
journal = {Networks},
volume = {1},
number = {2},
pages = {135--138},
doi = {https://doi.org/10.1002/net.3230010204},
url = {https://onlinelibrary.wiley.com/doi/abs/10.1002/net.3230010204},
eprint = {https://onlinelibrary.wiley.com/doi/pdf/10.1002/net.3230010204},
year = {1971}
}

@article{yuster2023tournamentinversion,
  title = {On Tournament Inversion},
  volume = {110},
  ISSN = {1097-0118},
  url = {http://dx.doi.org/10.1002/jgt.23251},
  DOI = {10.1002/jgt.23251},
  number = {1},
  journal = {Journal of Graph Theory},
  publisher = {Wiley},
  author = {Yuster,  Raphael},
  year = {2025},
  pages = {82--91},
  note = {\href{https://arxiv.org/abs/2312.01910}{\nolinkurl{arXiv:2312.01910}}}
}

@inproceedings{kovari1954,
  title={{On a problem of Zarankiewicz}},
  author={K{\H{o}}v{\'a}ri, Thomas and S{\'o}s, Vera T. and Tur{\'a}n, P{\'a}l},
  booktitle={Colloquium Mathematicum},
  volume={3},
  pages={50--57},
  year={1954},
  organization={Polska Akademia Nauk},
}

@article{reiman1958,
  title={{{\"U}ber ein problem von K. Zarankiewicz}},
  author={Reiman, Istvan},
  journal={Acta Mathematica Academiae Scientiarum Hungarica},
  volume={9},
  pages={269--273},
  year={1958},
  publisher={Springer}
}

@article{CAI199115,
title = {{NP-completeness of edge-colouring some restricted graphs}},
journal = {Discrete Applied Mathematics},
volume = {30},
number = {1},
pages = {15--27},
year = {1991},
issn = {0166-218X},
doi = {https://doi.org/10.1016/0166-218X(91)90010-T},
url = {https://www.sciencedirect.com/science/article/pii/0166218X9190010T},
author = {Leizhen Cai and John A. Ellis}
}

@article{inversion,
  doi = {10.48550/ARXIV.2212.09188},
  url = {https://arxiv.org/abs/2212.09188},
  author = {Aubian, Guillaume and Havet, Fr\'ed\'eric and H\"orsch, Florian and Klingelhoefer, Felix and Nisse, Nicolas and Rambaud, Cl\'ement and Vermande, Quentin },
  title = {Problems, proofs, and disproofs on the inversion number},
  journal = {arXiv preprint},
  year = {2022},  
  note = {\href{https://arxiv.org/abs/2212.09188}{\nolinkurl{arXiv:2212.09188}}}
}

@article{APSSW,
    author = {Alon, Noga and Powierski, Emil and Savery, Michael and Scott, Alex and Wilmer, Elizabeth},
    title = {Invertibility of Digraphs and Tournaments},
    journal = {SIAM Journal on Discrete Mathematics},
    volume = {38},
    number = {1},
    pages = {327--347},
    year = {2024},
    doi = {10.1137/23M1547135},
    note = {\href{https://arxiv.org/abs/2212.11969}{\nolinkurl{arXiv:2212.11969}}}
}

@article{BCH,
  author       = {J{\o}rgen Bang{-}Jensen and
                  Jonas Costa Ferreira da Silva and
                  Fr{\'{e}}d{\'{e}}ric Havet},
  title        = {On the inversion number of oriented graphs},
  journal      = {Discrete Mathematics \& Theoretical Computer Science},
  volume       = {23},
  number       = {2},
  year         = {2021},
  url          = {https://doi.org/10.46298/dmtcs.7474},
  doi          = {10.46298/DMTCS.7474},
  note = {\href{https://arxiv.org/abs/2105.04137}{\nolinkurl{arXiv:2105.04137}}}
}

@book{bang2009,
  title={Digraphs: theory, algorithms and applications},
  author={Bang-Jensen, J{\o}rgen and Gutin, Gregory Z.},
  year={2009},
  publisher={Springer-Verlag, London}
}

@article{BBBP10,
  title={Inversion dans les tournois},
  author={Belkhechine, Houmem and Bouaziz, Moncef and Boudabbous, Imed and Pouzet, Maurice},
  journal={Comptes Rendus Math{\'e}matique},
  volume={348},
  number={13-14},
  pages={703--707},
  year={2010},
  publisher={Elsevier},
  note={\href{https://arxiv.org/abs/1007.2103}{\nolinkurl{arXiv:1007.2103}}}
}

@article{bessyKernel,
title = {Kernels for feedback arc set in tournaments},
journal = {Journal of Computer and System Sciences},
volume = {77},
number = {6},
pages = {1071--1078},
year = {2011},
issn = {0022-0000},
doi = {https://doi.org/10.1016/j.jcss.2010.10.001},
author = {Stéphane Bessy and Fedor V. Fomin and Serge Gaspers and Christophe Paul and Anthony Perez and Saket Saurabh and Stéphan Thomassé},
note={\href{https://arxiv.org/abs/0907.2165}{\nolinkurl{arXiv:0907.2165}}},
}

@inproceedings{kenyon2007rank,
  title={How to rank with few errors},
  author={Kenyon-Mathieu, Claire and Schudy, Warren},
  booktitle={Proceedings of the thirty-ninth annual ACM symposium on Theory of computing},
  pages={95--103},
  year={2007}
}

@article{kovariCM3,
  title = {On a problem of {K}. {Z}arankiewicz},
  volume = {3},
  ISSN = {1730-6302},
  url = {http://dx.doi.org/10.4064/CM-3-1-50-57},
  DOI = {10.4064/cm-3-1-50-57},
  number = {1},
  journal = {Colloquium Mathematicum},
  publisher = {Institute of Mathematics,  Polish Academy of Sciences},
  author = {Kóvari,  Tamás and Sós, Vera Turán and Turán,  Pál},
  year = {1954},
  pages = {50--57}
}

\end{document}